\documentclass[11pt, draft]{amsart}
\usepackage{amssymb, amstext, amscd, amsmath, amssymb}
\usepackage{mathtools, xypic, paralist, color, dsfont, rotating}
\usepackage{verbatim}
\usepackage{enumerate,enumitem}
\usepackage{float}
\usepackage{setspace}
\usepackage{pifont}

\usepackage{tikz}

\floatstyle{boxed} 
\restylefloat{figure}

\numberwithin{equation}{section}
\usepackage[a4paper]{geometry}
\usepackage{quoting}
\quotingsetup{vskip=.1in}
\quotingsetup{leftmargin=.6in} 
\quotingsetup{rightmargin=.6in}

\let\OLDthebibliography\thebibliography
\renewcommand\thebibliography[1]{
  \OLDthebibliography{#1}
  \setlength{\parskip}{0pt}
  \setlength{\itemsep}{2pt plus 0.5ex}
}

\makeatletter
\def\@cite#1#2{{\m@th\upshape\bfseries%
[{#1\if@tempswa{\m@th\upshape\mdseries, #2}\fi}]}}
\makeatother
\theoremstyle{plain}
\newtheorem{theorem}{Theorem}[section]
\newtheorem{corollary}[theorem]{Corollary}
\newtheorem{proposition}[theorem]{Proposition}

\theoremstyle{definition}
\newtheorem{definition}[theorem]{Definition}
\newtheorem{example}[theorem]{Example}

\newtheorem{remark}[theorem]{Remark}

\theoremstyle{remark}

\mathtoolsset{centercolon}
  \newcommand{\A}{{\mathcal{A}}}
  \newcommand{\B}{{\mathcal{B}}}
  \newcommand{\C}{{\mathcal{C}}}

  \newcommand{\F}{{\mathcal{F}}}
  \newcommand{\G}{{\mathcal{G}}}
\renewcommand{\H}{{\mathcal{H}}}
  \newcommand{\I}{{\mathcal{I}}}
  \newcommand{\J}{{\mathcal{J}}}
  \newcommand{\K}{{\mathcal{K}}}
\renewcommand{\L}{{\mathcal{L}}}
  
  \newcommand{\N}{{\mathcal{N}}}
\renewcommand{\O}{{\mathcal{O}}}

\renewcommand{\S}{{\mathcal{S}}}
  \newcommand{\T}{{\mathcal{T}}}

  \newcommand{\X}{{\mathcal{X}}}
  \newcommand{\Y}{{\mathcal{Y}}}
  
\newcommand{\eps}{\varepsilon}
\def\al{\alpha}
\def\be{\beta}

\def\ga{\gamma}
\def\De{\Delta}
\def\de{\delta}
\def\ze{\zeta}

\def\la{\lambda}
\def\La{\Lambda}

\def\Om{\Omega}

\newcommand\vphi{\varphi}

\newcommand{\bC}{\mathbb{C}}

\newcommand{\bN}{\mathbb{N}}
\newcommand{\bT}{\mathbb{T}}
\newcommand{\bZ}{\mathbb{Z}}
\newcommand{\bR}{\mathbb{R}}
\newcommand{\fA}{{\mathfrak{A}}}

\newcommand{\Bx}{{\mathbf{x}}}
\newcommand{\By}{{\mathbf{y}}}
\newcommand{\Bz}{{\mathbf{z}}}

\newcommand{\foral}{\text{ for all }}
\newcommand{\qand}{\quad\text{and}\quad}

\newcommand{\qiff}{\quad\text{if and only if}\quad}
\newcommand{\qfor}{\quad\text{for}\quad}

\newcommand{\ca}{\mathrm{C}^*}

\newcommand{\cenv}{\mathrm{C}^*_{\textup{env}}}

\newcommand{\ol}{\overline}
\newcommand{\wt}{\widetilde}
\newcommand{\wh}{\widehat}

\newcommand{\ad}{\operatorname{ad}}

\newcommand{\alg}{\operatorname{alg}}

\newcommand{\cov}{\operatorname{c}}

\newcommand{\env}{\operatorname{\textup{env}}}

\newcommand{\fock}{\operatorname{F}}

\newcommand{\id}{{\operatorname{id}}}

\newcommand{\mt}{\emptyset}

\newcommand{\scv}{\operatorname{sc}}
\newcommand{\spn}{\operatorname{span}}

\newcommand{\sumoplus}{\operatornamewithlimits{\sum \strut^\oplus}}

\newcommand{\sca}[1]{\left\langle#1\right\rangle} 
\newcommand{\nor}[1]{\left\Vert #1\right\Vert} 
\newcommand{\bo}[1]{\mathbf{#1}} 
\newcommand{\quo}[2]{{\raisebox{.1em}{$#1$}\left/ \, \raisebox{-.1em}{$#2$}\right.}} 

\newcommand{\ses}[5]{
	\xymatrix{
		0 \ar[r] & #1 \ar[r]^-{#2} & #3 \ar[r]^-{#4} & #5 \ar[r] & 0
	}
}
\newcommand{\tes}[7]{
	\xymatrix@C=2cm@R=1.5cm{
		K_0\left(#1\right) \ar[r]^{#2} & K_0\left(#3\right) \ar[r]^{#4} & K_0\left(#5\right) \ar[d] \\
		K_1\left(#5\right) \ar[u] & K_1\left(#3\right) \ar[l]^{#7} & K_1\left(#1\right) \ar[l]^{#6}
	}
}

\begin{document}

\title[Product systems and their representations]{Product systems and their representations: an approach using Fock spaces and Fell bundles}

\author[E.T.A. Kakariadis]{Evgenios T.A. Kakariadis}
\address{School of Mathematics, Statistics and Physics\\ Newcastle University\\ Newcastle upon Tyne\\ NE1 7RU\\ UK}
\email{evgenios.kakariadis@newcastle.ac.uk}

\thanks{2010 {\it  Mathematics Subject Classification.} 46L08, 47L55, 46L05}

\thanks{{\it Key words and phrases:} Co-actions, Fock space, Fell bundles, product systems.}

\begin{abstract}
In this exposition we highlight product systems as the semigroup analogue of Fell bundles.
Motivated by Fock creation operators we extend the definition of Fowler's product systems over unital discrete left-cancellative semigroups, via both a concrete and an abstract characterization.
We next underline the importance of Fock spaces and Fell bundle methods for obtaining covariant representation results.
First, the strongly covariant representations of Sehnem, for product systems over group embeddable semigroups, find their analogue here.
Secondly, we give two solutions for the existence of the co-universal C*-algebra with respect to equivariant injective Fock-covariant representations in this context.
The first one uses an intertwining isometry technique to realize it as the C*-envelope of the normal cosystem on the Fock tensor algebra.
The second uses a Fell bundle quotient technique to realize it as the reduced C*-algebra of Sehnem's strong covariance relations.
This C*-algebra is boundary for a larger class of equivariant injective representations.
\end{abstract}

\maketitle

\section{Introduction}


\noindent
\textbf{Framework.}
A boundary quotient is a terminal object for a class of representations.
In most of the cases boundary quotients are manifestations of corona sets and further our understanding beyond the ``classical'' universe of finite dimensional algebras and their norm-limits.

They have been a continuous source of inspiration in particular in the operator algebras theory.
First of all, they provide key C*-constructions in the study of geometric and topological objects, e.g., semigroups, graphs, dynamical systems etc.
Dating as back as the classification of factors by Murray and von Neumann in the 1930's and 1940's, they have been used successfully for detecting phase transitions of C*-measurements, such as K-theory, equilibrium states, finite dimensional approximations, or quasi-diagonal representations, and help reveal properties of the structures that may not be apparent intrinsically.
A second manifestation of boundary representations arises through the notions of the \v{S}ilov and the Choquet boundaries.
Motivated by the interaction with convexity and approximation theory \cite{CM63} on (possibly nonselfadjoint) function algebras, Arveson foresaw a noncommutative analogue in his ground-breaking work \cite{Arv69}, later established by Hamana \cite{Ham79}. 
A multi-faceted combination of the selfadjoint and nonselfadjoint viewpoints has been recently confirmed, with fruitful applications in group theory \cite{BKKO17, KK17}, noncommutative geometry \cite{CvS+}, and noncommutative convexity \cite{EH19}.

Herein we study product systems over a unital discrete group-embeddable semigroup $P \subseteq G$, and their representations.
Explored for more than 30 years, this construct models a great number of operator algebras, including those associated with graphs (of rank 1 or higher), dynamical systems of several flavours (reversible and irreversible), semigroups, and Nica-covariant product systems.
The identification of the co-universal C*-object has been of major importance and is at the base of many other developments, e.g., C*-correspondences \cite{DKPS10, KK06, Kat04, MS98, Pim95}, product systems \cite{AM15, CLSV11, DFK17, DKKLL20, DK20, Fow99, Fow02, KS16, Seh18, SY10}, and semigroups and higher rank graphs \cite{KKLL21, KK06g, Li12, Nic92, RSY04, RS99}.
The $P = \bZ_+$ case has been quite accommodating in this respect, e.g., see \cite{DK18} for a discussion that we will avoid repeating here.
Applications beyond $\bZ_+$ include results on the C*-structure, Takai duality, and the Hao-Ng problem, e.g., \cite{AM15, DFK17, DKPS10, DKKLL20, DK20, Fle18, HN08, Kak19, KKLL21, Kat20, KP00, KL19b, KS16, RRS17, Sch15}, to mention but a few.

A long-standing question has been the identification of the boundary quotient with the (plain) C*-envelope of the Fock tensor algebra.
The breakthrough connection with co-universality was established by Dor-On--Katsoulis \cite{DK18}, where they show that all objects (the boundary quotient, the C*-envelope and the co-universal object) coincide in the abelian lattice case.
These ideas were further explored by Dor-On--Kakariadis--Katsoulis--Laca--Li \cite{DKKLL20} for the right LCM case and by Kakariadis--Katsoulis--Laca--Li \cite{KKLL21} for the semigroup algebras case, by introducing the C*-envelope of a cosystem, where it was noted that further advancements would require building a conditional expectation on the C*-envelope.
Recently, Sehnem \cite{Seh21} answered this conjecture and provided the full solution to this long-standing problem.

Our main goal here has been to amplify the Fock space and the Fell bundle techniques of \cite{CLSV11, DKKLL20, Exe97, Fow02, KKLL21, Rae92, Seh18} to a more general context than that of Fowler, and highlight product systems as the semigroup analogue of Fell bundles.
Motivated by a concrete picture, we consider product systems over any unital discrete left-cancellative semigroup, and use the Fock space as a model of covariance.
We also provide an equivalent definition of abstract product systems.
We show that the equivariant injective Fock-covariant representations of product systems admit a terminal object via two incarnations: as the C*-envelope of the cosystem on the Fock tensor algebra from \cite{DKKLL20}, and as the reduced C*-algebra of the strong covariance Fell bundle from \cite{Seh18}.
Working backwards we identify a super-class of covariant representations that admits the same terminal object and argue why this is the furthest one goes.


\vspace{4pt}

\noindent
\textbf{Fell bundles.}
Hilbertian representations play an important role in the study of algebraic structures.
One of the cornerstone examples is the Gel'fand--Raikov Theorem that ``identifies'' a locally compact topological group with its unitary representations.
The idea of \emph{quantized/quantum functional analysis} (as it was termed in the beginning) has been the driving force for constructing C*-algebras of groups, dynamical systems, graphs etc.
Similar to the group case, at one end there is a maximal/universal C*-algebra that emits to all possible representations, and at the other end there is a minimal/reduced C*-algebra that receives from all possible representations that admit a grading and are injective on the initial generators.
Exel \cite{Exe97} provided a great generalization at the level of Fell bundles over a discrete group $G$, showing that the left regular representation is terminal in this class.
His seminal work unifies and covers a great number of constructions, including that of graphs and crossed products of (partial) dynamics.

\vspace{4pt}

\noindent
\textbf{Product systems.}
Product systems aim in modelling a semigroup analogue of Fell bundles.
They appeared first in the work of Arveson \cite{Arv89} disguised under duality, and considered later by Dinh \cite{Din91} for discrete semigroups of $\bR_+$.
Fowler \cite{Fow99, Fow02} formally defined product systems $X = \{X_p\}_{p \in P}$ and proceeded to an in-depth study when $(G, P)$ is a quasi-lattice, that inspired a great number of subsequent works.
They have been used extensively for encoding C*-constructs with the list of examples being rather long to include here.
The interested reader may consult Michael Frank's detailed Literature List on Hilbert modules (available online).
It may be instructive, however, to give the analogies with the prototypical example of graphs \cite{CK80}.
Given a finite directed graph $\G$ there are two C*-models for the reconstruction of vertices:
the Toeplitz--Cuntz--Krieger C*-algebra $\T_\G$ (from emitting edges) and the Cuntz--Krieger C*-algebra $\O_\G$ (from emitting and receiving edges).
They both admit a coaction of $\bZ$ by measuring the length of paths (with adjoints contributing with a negative length).
A third operator algebra, nonselfadjoint, arises by considering the $\bZ_+$-semigroup: we write $\A_\G$ for the algebra generated by $\G$ in $\T_\G$.
We thus have the following analogies with the notions we are going to use: 

\vspace{-2pt}
\begin{enumerate}
\item The C*-algebra $\T_\G$ represents the Fock-covariant representations.
\item The operator algebra $\A_\G \subseteq \T_\G$ represents the Fock tensor algebra.
\item The C*-algebra $\O_\G$ represents the strong covariance relations.
\end{enumerate}

\vspace{-2pt}
\noindent
The C*-algebra $\O_\G$ is co-universal in two ways.
If $\Phi$ is a $\bZ$-equivariant $*$-representation of $\T_\G$ that is injective on the vertices, then there exists a canonical $*$-epimorphism from $\ca(\Phi)$ onto $\O_\G$.
On the other hand, if $\phi$ is completely isometric representation of $\A_\G$ (that inherits the coaction of $\bZ_+$), then there exists a canonical $*$-epimorphism from $\ca(\phi)$ onto $\O_\G$.
Hence $\O_\G$ is both a C*-boundary quotient and an $\A_\G$-boundary representation in the sense of Arveson.


\vspace{4pt}

\noindent
\textbf{Nica-covariance.}
Fowler was greatly inspired by Nica's work \cite{Nic92} on quasi-lattices.
It had been known that, unlike to the group case, taking plain isometric representations leads to intractable objects.
For example, the universal isometric C*-algebra of $\bN \times \bN$ is not even nuclear \cite{Mur87}.
A richer structure is at hand when one considers the relations reflected in the Fock representation of a quasi-lattice order $(G, P)$, known as Nica-covariance.
In order to deal with this, Fowler imposed several axioms on the product system such as being compactly aligned and that the ``product'' $X_p \cdot X_q$ is dense in $X_{pq}$ for all $p, q \in P$.
A major consequence of this setup is a Wick ordering, i.e., the ``monomials'' $X_p \cdot X_q^*$ densely span the generated C*-algebra.
The same picture is acquired by Kwa\'{s}niewski--Larsen \cite{KL19b} when $P$ is a right LCM-semigroup (see also \cite{DKKLL20}).
In practice, compact alignment on right LCM-semigroups is the furthest one can go and still admit a Wick ordering.


\vspace{4pt}

\noindent
\textbf{Fock covariance.}
Nica-covariance may fail for general semigroups, but one can still use the Fock representation as the prototype for covariant-type relations.
This has been noticed by Cuntz--Deninger--Laca \cite{CDL13} for $P = R \rtimes R^\times$.
In his ground-breaking work Li \cite{Li12} introduced and studied semigroup representations that, apart from the semigroup structure, they remember the principal ideals and their intersections, i.e., the \emph{constructible} ideals.
This model was further examined by Kakariadis--Katsoulis--Laca--Li \cite{KKLL21} in connection also with inverse semigroup realizations of Norling \cite{Nor14}.
Recently Laca--Sehnem \cite{LS21} identified the correct covariant relations that place the Fock representation of a semigroup as the reduced C*-algebra of a natural Fell bundle.
The constructible and the Fock model coincide only under an independence condition \cite{KKLL21, LS21}.
In the case of product systems it has been observed in \cite[Proposition 4.3]{DKKLL20} that the Nica-Toeplitz C*-algebra is maximal for the Fock bundle in the case of right LCM-semigroups.
Interestingly the Nica-covariance and compact alignment for right LCM-semigroups are connected in a solid way as shown in \cite[Proposition 3.2]{Kat20} and \cite[Proposition 4.7]{KKLL21b}.


\vspace{4pt}

\noindent
\textbf{Strong covariance.}
The graph C*-algebra $\O_\G$ provides a corona context for the graph by moding out appropriate compact operators.
Following a series of works, Fowler--Laca--Raeburn \cite{FLR00} defined $\O_\G$ for general graphs.
Later Katsura \cite{Kat04} established the covariant relations that give the correct analogue of a Cuntz-type C*-algebra for single C*-correspondences.
An extensive effort for establishing the strong covariance relations for general product systems has been put since then.
The goal had been to identify a C*-algebra that: 

\vspace{-2pt}
\begin{enumerate}
\item[(a)] admits an injective copy of the co-efficient C*-algebra $A$; 
\item[(b)] inherits a coaction of $G$; and 
\item[(c)] any representation that is injective on $A$ is faithful on its fixed point algebra.
\end{enumerate}

\vspace{-2pt}
\noindent
For \emph{regular} product systems this was met with much success \cite{AM15, DKPS10, Fow02, KS16}.
Moving forwards, Sims--Yeend \cite{SY10} introduced the augmented Fock representation as a candidate for quasi-lattices, and Carlsen--Larsen--Sims--Vittadello \cite{CLSV11} showed that it is an effective model in many non-regular cases.
At the same time they showed that this setup has its limitations and there are examples of semigroups that are not encountered, e.g., for $P = \left( (\bZ_+\setminus \{0\}) \times \bZ \right) \bigcup \left( \{0\} \times \bZ_+ \right)$ in $\bZ \times \bZ$ with the lexicographical order \cite[Example 3.16]{SY10}, or when $P$ is not directed \cite[Example 3.9]{CLSV11}.
Recently, Sehnmen \cite{Seh18} provided the universal construction $A \times_X P$ of strong covariance relations for general product systems.
This breakthrough step encompasses previous results and works for Fowler's product systems with no restrictions on the semigroup whatsoever.
Sehnem terms $A \times_X P$ as the ``covariance C*-algebra of $X$''.
However in view of the different covariant relations we treat here and since the representations of $A \times_X P$ are termed by Sehnem as strongly covariant, we will refer to $A \times_X P$ as the \emph{strong covariance C*-algebra of $X$}.


\vspace{4pt}

\noindent
\textbf{Boundary quotients.}
Strongly covariant quotients sometimes are referred to as boundary quotients, since by construction they produce a non-commutative boundary in a corona universe.
They have been shown to give a boundary quotient in several constructions in two ways: as the \v{S}ilov boundary of the nonselfadjoint Fock tensor algebra (and thus its C*-envelope in the sense of Arveson), and as the co-universal representation of $X$ with respect to equivariant Fock-covariant representations that are injective on $A$.
At the level of C*-correspondences the C*-envelope question was answered by Katsoulis--Kribs \cite{KK06}, while co-universality of the Cuntz-Pimsner algebra was proven by Katsura \cite{Kat07} by a parametrization of the equivariant ideals of the Fock representation.
Concerning compactly aligned product systems, Carlsen--Larsen--Sims--Vittadello \cite{CLSV11} considered the co-universality question for quasi-lattices, with additional assumptions on $X$.
The C*-envelope question for C*-dynamics over $\bZ_+^d$ was answered by Davidson--Fuller--Kakariadis \cite{DFK17}, who used the C*-envelope machinery to acquire the form of the co-universal C*-algebra.
Recently, Dor-On--Katsoulis \cite{DK20} used the C*-envelope to answer the co-universality question for abelian lattices, providing solid evidence of a novel connection between the two problems.
Lately, Dor-On--Kakariadis--Katsoulis--Laca--Li \cite{DKKLL20} resolved the issue at the generality of right LCM-semigroups, which has been the latest to date context for product systems theory.
A pivotal step in their work is the initiation of a C*-envelope theory for cosystems on possibly nonselfadjoint operator algebras.
In parallel, Kakariadis--Katsoulis--Laca--Li \cite{KKLL21} have attained boundary quotient results for trivial fibers over any group-embeddable semigroup.
The essence of \cite{DKKLL20, KKLL21} is that they establish the compatibility between the two uses of the ``boundary'' term, as in the C*-theory and as in the nonselfajoint theory.
The methods of \cite{DKKLL20, KKLL21} rely on a blend of Fock space and Fell bundle techniques, and form a cornerstone for the current work.


\vspace{4pt}

\noindent
\textbf{Gauge-invariant uniqueness theorems.}
A gauge-invariant-uniqueness-theorem (GIUT) provides necessary and sufficient conditions for an equivariant representation of $A \times_X P$ to be faithful.
This can be derived for example by property (c) of strong covariance for amenable groups $G$.
Sometimes co-universality results are termed as GIUT's (or the other way around).
However, the GIUT for strong covariance is weaker than co-universality, as it refers to a subclass of representations of $X$.
Even more, any relative Cuntz-Pimsner algebra for $P= \bZ_+$ satisfies an appropriate GIUT \cite{Kak14a}, while one cannot hope for a GIUT for $A \times_X P$ without imposing amenability.


\vspace{4pt}

\noindent
\textbf{Main results.}
We therefore arrive to the following questions:

\vspace{-2pt}
\begin{enumerate}[leftmargin=30pt]
\item[(I)] Fowler's product systems impose that $X_p \cdot X_q$ be dense in $X_{pq}$.
Although the inclusion $X_p \cdot X_q$ in $X_{pq}$ is easy to check, the density requirement adds an extra step for associating a C*-construct with a product system.
Can this be relaxed, and how?
\item[(II)] The (unconditional) universal C*-algebra of a product system may lead to intractable objects, even for abelian orders.
A rich structure is prominent for the universal C*-algebra that is related to the Fell bundle arising in the Fock representation has been observed in \cite{DKKLL20, KKLL21b} for right LCM-semigroups and in \cite{LS21} for semigroups.
Different covariant models give in principle non-isomorphic universal algebras.
However, substantial progress can be made without those completely being identified, as for example in \cite{CDL13} and \cite{Li12, Nor14} for semigroups.
What are the key algebraic properties from the Fock-covariant model?
\item[(III)] The co-universal object for equivariant injective Fock-covariant representations of $X$ has been computed for compactly aligned product systems over a right LCM-semigroup \cite{DKKLL20} and for trivial product systems over any group-embeddable semigroup \cite{KKLL21}, in direct connection to \cite{CLSV11, DK20, Seh18}. 
In some cases, the same strongly covariant object is co-universal for a wider class of representations \cite{KKLL21}.
What is the largest covariant model that accommodates such a result?
\end{enumerate}

\vspace{-2pt}
\noindent
Taking motivation from \cite{DKKLL20, KKLL21, KKLL21b, LS21}, we consider an expanded notion of product systems, which includes that of Fowler \cite{Fow02}, by relaxing the condition in item (I).
We show that this class is still accommodating for key results concerning their covariant representations in relation to item (II).
Towards this end we show that the techniques of \cite{CLSV11, DKKLL20, Exe97, KKLL21, Seh18} find their analogue here under the appropriate adjustments.
As a consequence, we obtain the solution to the co-universality problem of item (III) without assuming extra structure on $X$ or $P$, gathering previous co-universality results, for example those that appear in \cite{AM15, CLSV11, DFK17, DKKLL20, DK20, Exe97, Fow02, KKLL21, KK06g, KK06, Kat04, KP00, KS16, Li12, MS98, Pim95, RRS17, RS99, Seh18, SY10}.


\vspace{4pt}

\noindent
\textbf{Discussion of results.}
Section 2 is devoted to preliminaries.
We describe the main results from \cite{DKKLL20, Exe97} that we are going to use here.

In Section 3 we revisit the definition of a product system, and we first consider concrete families $X = \{X_p\}_{p \in P}$ of operator spaces in a common $\B(H)$.
We require that the inclusions $X_p \cdot X_q \subseteq X_{pq}$ are satisfied, as in the super-product systems setting of Margetts--Srinivasan \cite{MS13} and Shalit--Skeide \cite{SS20}.
We next make use of the left-cancellation of $P$ to induce an ``adjoint behaviour'' and obtain \emph{adjointable} left creation operators on the Fock space, by declaring that $A := X_e$ is a C*-algebra and that $X_p^* \cdot X_{pq} \subseteq X_q$.
These alterations still make each $X_p$ a C*-correspondence over $A$ with the structure endowed by $\B(H)$.
At the same time, we provide an abstract characterization of a product system, which is useful for algebraic constructs.
We show that there is no distinction between ``concrete'' and ``abstract'' in terms of the related operator algebras.
Next we study a generalized version of cores and we give some preliminary results on the fixed point algebra of equivariant representations.

In Section 4 we use the Fock representation to locate the key properties for a rich algebraic structure.
First we use the coaction of $G$ on the Fock representation to identify a universal \emph{Fock-covariant} C*-algebra $\T_{\cov}^{\fock}(X)$, whose Fell bundle has $\T_\la(X)$ as its reduced C*-algebra.
We also identify the injective representations of $X$ that descend to injective Fock-covariant representations, which we refer to simply as \emph{covariant}.
In particular, we show that if an equivariant injective representation factors through a Fock-covariant one then it needs to be covariant.
Secondly we extend the results of Sehnem \cite{Seh18} and show that there is a direct analogue of the universal \emph{strong covariance} C*-algebra $A \times_X P$ in this context.
Most of the arguments of \cite{Seh18} pass through here, with the difference that we unconditionally show that there are canonical $*$-epimorphisms
\[
\T(X) \longrightarrow \T_{\cov}^{\fock}(X) \longrightarrow A \times_X P,
\] 
that fix $X$.
We have included plenty of details from \cite{Seh18} and \cite{DKKLL20} here in order to make clear how the arguments extend to our setting.
The $G$-Fell bundle in $A \times_X P$ defines also the \emph{reduced} strong covariance C*-algebra $A \times_{X, \la} P$ that will be of great importance and was first considered in \cite{DKKLL20}.
The C*-algebras of this section are independent of the group embedding of $P$ (and this why $G$ is omitted from their notation).

In Section 5 we show that the class of equivariant injective covariant representations of $X$ admits a terminal object.
One key point is that we can reduce to the case of injective Fock-covariant representations.
In the first approach we give a nonselfadjoint proof through the C*-envelope $\cenv(\T_\la(X)^+, G, \de)$ of the cosystem on the algebra $\T_\la(X)^+$ generated by $A$ and $X$ in $\T_\la(X)$, first introduced in \cite{DKKLL20}.
Towards this end we prove a Fell's absorption principle for $\T_\la(X)^+$ by any equivariant representation of $\T_\la(X)$ that is injective on $A$.
Thus such representations generate C*-covers for $\T_\la(X)^+$, and the strategy of \cite{DKKLL20} applies to give that $\cenv(\T_\la(X)^+, G, \de)$ is co-universal for the equivariant injective covariant representations of $X$.
Due to the normality of the coaction we deduce that it coincides with $A \times_{X, \la} P$.
Alternatively, we give a C*-algebraic approach of this result from Carlsen--Larsen--Sims--Vittadello \cite{CLSV11}, inspired by \cite{Kat85, Rae92}.
We note that the algebraic structure of strong covariance combines with the general Fell bundle arguments of \cite{CLSV11}, and gives that $A \times_{X, \la} P$ is the required co-universal object.
As the C*-envelope of the cosystem is in the class of representations and inherits a normal coaction, it then coincides with $A \times_{X, \la} P$.
We close with some immediate consequences of the co-universality results, including a GIUT for $A \times_{X, \la} P$.
They encompass a number of results in the literature, highlighting in the most emphatic way the ground-breaking significance of Sehnem's strong covariance relations \cite{Seh18}, and the angles taken in \cite{DKKLL20, KKLL21, KKLL21b}.

\vspace{4pt}

\noindent
\textbf{Notes and remarks.}
The author would like to acknowledge that in the context of the joint research with Elias Katsoulis, Marcelo Laca and Xin Li leading to \cite{KKLL21}, he had access to a draft from Marcelo Laca that, among other considerations relevant to \cite{KKLL21}, contained: 1) A contribution attributed to, and shared with the consent of, Camila F. Sehnem consisting of a proof that every gauge-equivariant $*$-homomorphism of $\mathcal{T}_\lambda(X)$ that is injective on $A$ is completely isometric on the tensor algebra $\mathcal{T}_\lambda(X)^+$. 
This can now be found in Proposition 3.1 and Corollary 3.5 of \cite{Seh21}, cf.~Propositions \ref{P:cocover nsa} and \ref{P:cocover cstar} below. 
2) A discussion of the ``domain of attraction" of co-universality of the cosystem C*-envelope of the semigroup algebra, due to Laca and Sehnem. 
This included a proof that the C*-envelope of the cosystem arising from the tensor algebra of a submonoid $P$ of a group is co-universal for equivariant isometric representations $T$ of $P$ such that $\ca(T)$ has a common nontrivial equivariant quotient with $\mathcal{T}_u(P)$, and the suggestion that a sufficient condition for this is that the quotient of the reduced  C*-algebra of the Fell bundle of $\ca(T)$ by the  ideal generated by the relations (T1)-(T4) of \cite{LS21} is nontrivial, cf.~Main Results (III) above. 
The idea of modeling a universal Toeplitz C*-algebra for a product system on its Fock representation emerged in the context of the aforementioned collaboration  in conversations with Elias Katsoulis, and also with Marcelo Laca in relation to a joint project of his with Camila F. Sehnem building upon \cite{LS21},  cf.~Main Results (II) above.


\vspace{4pt}

\noindent
\textbf{Acknowledgements.}
The author acknowledges support from EPSRC through the ``Operator Algebras for Product Systems'' programme (ref.\ EP/T02576X/1).
The author would like to thank Elias Katsoulis, Marcelo Laca and Xin Li for the motivating discussions during the production of \cite{KKLL21,KKLL21b}.
The author would like to dedicate this work to Nellie (and our random walks and ---occasionally--- talks), and to Neroccio di Bartolomeo de' Landi for the inspiration.

\section{Operator algebras and their coactions}

\subsection{Operator algebras}

If $\X$ and $\Y$ are subspaces of a normed algebra, then we write $[\X \cdot \Y]$ for the norm-closure of the linear span of the elements $x \cdot y$ with $x \in \X$ and $y \in \Y$.
Our main references for operator spaces theory are the monographs \cite{BL04, Pau02}.
For this paper an operator algebra is a subalgebra of some $\B(H)$ for a Hilbert space $H$.
We will write $\otimes$ for the spatial tensor product between operator spaces.

We say that $(C, \iota)$ is a \emph{C*-cover} of $\fA$ if $\iota \colon \fA \to C$ is a completely isometric representation with $C = \ca(\iota(\fA))$.
The \emph{C*-envelope} $\cenv(\fA)$ of $\fA$ is a C*-cover $(\cenv(\fA), \iota)$ with the following co-universal property:
if $(C', \iota')$ is a C*-cover of $\fA$, then there exists a (necessarily unique) $*$-epimorphism $\Phi \colon C' \to \cenv(\fA)$ such that $\Phi(\iota'(a)) = \iota(a)$ for all $a \in \fA$.
Arveson predicted the existence of the C*-envelope in \cite{Arv69} which he computed for a variety of operator algebras.
Later Hamana \cite{Ham79} removed all assumptions and proved the unconditional existence of injective envelopes.
The C*-envelope is the C*-algebra generated by $\fA$ in its injective envelope once this is endowed with the Choi--Effros C*-structure. 

The basic example of a C*-envelope arises in the case of uniform algebras: the C*-envelope of a uniform algebra is formed by the continuous functions on its \v{S}ilov boundary.
For the non-commutative analogue of this result consider $\fA \subseteq \ca(\fA)$.
An ideal $\I \lhd \ca(\fA)$ is called \emph{boundary} if the quotient map $\ca(\fA) \to \ca(\fA)/\I$ is completely isometric on $\fA$.
The \emph{\v{S}ilov ideal} $\I_s$ is a boundary ideal that contains all the boundary ideals of $\fA$.
The existence of the C*-envelope implies the existence of the \v{S}ilov ideal; in particular, it follows that $\cenv(\fA)$ is canonically isomorphic to $\ca(\fA)/\I_s$.

\subsection{Coactions on operator algebras}

All groups we consider are discrete.
For a discrete group $G$ we write $u_g$ for the generators of the universal group C*-algebras $\ca_{\max}(G)$, and $\la_g$ for the generators of the left regular representation $\la \colon \ca_{\max}(G) \to \ca_\la(G)$.
We write $\chi$ for the character on $\ca_{\max}(G)$.

The notion of a cosystem for a (possibly nonselfajoint) operator algebra $\fA$ was introduced in \cite{DKKLL20}.
A \emph{coaction of $G$ on $\fA$} is a completely isometric representation $\de \colon \fA \to \fA \otimes \ca_{\max}(G)$ such that the induced subspaces
\[
\fA_g := \{a \in \fA \mid \de(a) = a \otimes u_g\}
\]
densely span $\fA$.
Therefore $\de$ satisfies the coaction identity
\[
(\de \otimes \id_{\ca_{\max}(G)}) \de = (\id_{\fA} \otimes \De) \de.
\]
If, in addition, the map $(\id \otimes \la) \de$ is injective, then the coaction $\de$ is called \emph{normal}. 
The triple $(\fA, G, \de)$ is called a \emph{cosystem}.
A map $\Phi \colon \fA \to \fA'$ between two cosystems $(\fA, G, \de)$ and $(\fA', G, \de')$ is said to be \emph{$G$-equivariant} if $\de' \Phi = (\Phi \otimes \id) \de$; equivalently $\Phi(\fA_g) \subseteq \fA'_g$ for all $g \in G$.
We will be using the term \emph{equivariant} when the group $G$ is understood.

This definition covers that of Quigg for full coactions on C*-algebras \cite{Qui96}.
In \cite{Qui96} the coactions are assumed to be non-degenerate instead of the fibers densely spanning the C*-algebra, but for discrete groups these two properties are equivalent.

\begin{remark}\label{R:nd cis}
Suppose that $(\fA, G, \de)$ is a cosystem and that $\de$ extends to a $*$-homomorphism $\de \colon \ca(\fA) \to \ca(\fA) \otimes \ca_{\max}(G)$ that satisfies the coaction identity
\[
(\de \otimes \id) \de(c) = (\id \otimes \De) \de(c) \foral c \in \ca(\fA).
\]
By \cite[Remark 3.2]{DKKLL20} then $\de$ is automatically non-degenerate on $\ca(\fA)$ in the sense that
\[
\ol{\de(\ca(\fA)) \left[\ca(\fA) \otimes \ca_{\max}(G)\right]} = \ca(\fA) \otimes \ca_{\max}(G).
\]

In particular, the coactions on operator algebras arise as restrictions of coactions on C*-algebras.
Recall (for example from \cite[Proposition 2.4.2]{BL04}) that $(\ca_{\max}(\fA), j)$ is the C*-cover of $\fA$ with the following universal property: every completely contractive representation $\Phi \colon \fA \to \B(H)$ extends to a $*$-representation $\wt{\Phi} \colon \ca_{\max}(\fA) \to \B(H)$.
Therefore, if $(\fA, G, \de)$ is a cosystem, then we may consider the $*$-representation extension $\wt{\de}$ on $\ca_{\max}(\fA)$ by using the completely isometric homomorphism:
\[
\ca_{\max}(\fA) \supseteq j(\fA) \longrightarrow \fA \longrightarrow \fA \otimes \ca_{\max}(G) \longrightarrow j(\fA) \otimes \ca_{\max}(G) \subseteq \ca_{\max}(\fA) \otimes \ca_{\max}(G).
\]
Since $\id \otimes \chi$ is a left inverse for $\wt{\de}$, we get that $\wt{\de}$ is faithful.
Moreover, we have that each fibre
\[
[\ca_{\max}(\fA)]_g := \{x \in \ca_{\max}(\fA) \mid \wt{\de}(x) = x \otimes u_g\}
\]
contains all elements from $\fA_g$ and $(\fA_{g^{-1}})^*$, and all elements of the form
\[
a_{g_1}^{\eps_1} a_{g_2}^*a_{g_3} \cdots a_{g_{n-1}} a_{g_n}^* a_{g_{n+1}}^{\eps_2} \qfor g = g_1^{\eps_1} g_2^{-1} g_3 \cdots g_{n-1} g_n^{-1} g_{n+1}^{\eps_2}; n \in \bZ_+; \eps_{1}, \eps_2 \in \{0,1\};
\]
with the understanding that, if $\eps_i = 0$, then we omit that element from the product.
By construction $\ca_{\max}(\fA)$ is densely spanned by these monomials and therefore $\wt{\de}$ is a coaction.
\end{remark}

A \emph{reduced} coaction on $\fA$ consists of a family of spaces $\{\fA_g\}_{g \in G}$ that densely span $\fA$ and a completely isometric homomorphism $\de_\la \colon \fA \to \fA \otimes \ca_\la(G)$ such that
\[
\de_\la(a) = a \otimes \la_g \foral a \in \fA_g.
\] 
By using the Fell absorption principle, it is shown in \cite[Proposition 3.4]{DKKLL20} that in this case there is a coaction $\de$ such that $(\id \otimes \la) \de = \de_\la$, i.e., reduced coactions lift to normal coactions.

In \cite{DKKLL20} the authors introduce the notion of the C*-cover and of the C*-envelope for a cosystem.
A triple $(C, \iota, \de_C)$ is called a \emph{C*-cover} for a cosystem $(\fA, G, \de)$ if $(C, \iota)$ is a C*-cover of $\fA$ and $\de_C \colon C \to C \otimes \ca_{\max}(G)$ is a coaction on $C$ such that the following diagram
\[
\xymatrix{
\fA \ar[rr]^{\iota} \ar[d]^{\de} & & C \ar[d]^{\de_C} \\
\fA \otimes \ca_{\max}(G) \ar[rr]^{\iota \otimes \id} & & C \otimes \ca_{\max}(G)
}
\]
commutes.
The \emph{C*-envelope} of $(\fA, G, \de)$ is a C*-cover $(\cenv(\fA, G, \de), \iota, \de_{\env})$ such that: for every C*-cover $(C', \iota', \de')$ of $(\fA, G, \de)$ there exists a $*$-epimorphism $\Phi \colon C' \to \cenv(\fA, G, \de)$ that fixes $\fA$ and intertwines the coactions, i.e., the diagram
\[
\xymatrix{
\iota'(\fA) \ar[rrr]^{\de'} \ar[d]^{\Phi} & & & C' \otimes \ca_{\max}(G) \ar[d]^{\Phi \otimes \id} \\
\iota(\fA) \ar[rrr]^{\de_{\env}} & & & \cenv(\fA, G, \de) \otimes \ca_{\max}(G)
}
\]
is commutative on $\iota'(\fA)$, and thus extends to a commutative diagram on $C'$.

One of the main results of \cite{DKKLL20} is that the C*-envelope of a cosystem $(\fA, G, \de)$ exists and that it links to the usual C*-envelope in the following way.
Fix the embedding $i \colon \fA \to \cenv(\fA)$.
Then the C*-algebra $\ca(i(a_g) \otimes u_g \mid g \in G)$ is a C*-cover by considering the representation:
\[
\xymatrix@C=2cm{
\fA \ar[r]^{\de \phantom{ooo} } & \fA \otimes \ca_{\max}(G) \ar[r]^{i \otimes \id \phantom{oooo} } \ar[r] & \cenv(\fA) \otimes \ca_{\max}(G).
}
\]
We can then endow it with the coaction $\id \otimes \De$, and in \cite[Theorem 3.8]{DKKLL20} it is shown that
\[
(\cenv(\fA, G, \de), \iota, \de_{\env}) \simeq (\ca(i(a_g) \otimes u_g \mid g \in G), (i \otimes \id) \de, \id \otimes \De).
\]
If $\de$ is normal, then Fell's absorption principle implies that the induced coaction on $\cenv(\fA, G, \de)$ is normal as well \cite[Corollary 3.9]{DKKLL20}.

If $G$ is abelian then the C*-envelope $\cenv(\A, G, \de)$ of a cosystem $(\A, G, \de)$ coincides with the C*-envelope $\cenv(\A)$, i.e., when $G$ is abelian then the C*-envelope inherits the coaction from $\A$.
Indeed, coactions over abelian groups $G$ are equivalent to point-norm continuous actions $\{\be_\ga \mid \ga \in \wh{G}\}$ of the dual group $\wh{G}$.
Since every $\be_\ga$ is a completely isometric automorphism it extends to the C*-envelope.
It is unknown if this is the case for amenable groups in general.

\subsection{Semigroups}

For a unital discrete left-cancellative semigroup $P$ we write
\[
\ca_\la(P) := \ca(V_p \mid p \in P)
\qand
\A(P) := \ol{\alg}\{V_p \mid p \in P\},
\]
for the left-creation (isometric) operators given by
\[
V_p \colon \ell^2(P) \to\longrightarrow \ell^2(P) ; \de_s \mapsto \de_{ps}.
\]
The restriction to the diagonal $\ell^\infty(P)$ induces a faithful conditional expectation on $\ca_\la(P)$.

We will require some facts about the right ideals of $P$ from \cite{Li12}.
For a set $X \subseteq P$ and $p \in P$ we write
\[
pX:=\{px \mid x \in X\}
\qand
p^{-1}X := \{y \in P \mid py \in X\}.
\]
Note here that by definition $p^{-1} P = P$.
We write $\J$ for the smallest family of right ideals of $P$ containing $P$ and $\mt$ that is closed under left multiplication, taking pre-images under left multiplication (as in the sense above) and finite intersections, i.e.,
\[
\J := \left\{ \bigcap\limits_{j=1}^N q_{j, n_j}^{-1} p_{j, n_j} \dots q_{j, 1}^{-1} p_{j, 1} P \mid N, n_j \in \bZ_+; p_{j,k}, q_{j,k} \in P \right\} \cup \{\mt\}.
\]
The right ideals of $P$ in $\J$ are called \emph{constructible}.
It follows from \cite[Lemma 3.3]{Li12} that 
\[
q_n^{-1} p_n \dots q_1^{-1} p_1 p_1^{-1} q_1 \dots p_n^{-1} q_n X = (q_n^{-1} p_n \dots q_1^{-1} p_1 P) \cap X
\]
for every $p_i, q_i \in P$ and every subset $X$ of $P$.
Thus the set of constructible ideals is automatically closed under finite intersections, i.e.,
\[
\J = \left\{ q_n^{-1} p_n \dots q_1^{-1} p_1 P \mid n \in \bZ_+; p_i, q_i \in P \right\} \cup \{\mt\}.
\]
We will write $\Bx, \By, \Bz$ etc.\ for the elements in $\J$.
For a set $X \subseteq P$ we will write $E_{[X]}$ for the projection on $\sca{\de_p \mid p \in X}$.
It follows that $E_{[\Bx]} E_{[\By]} = E_{[\Bx \cap \By]}$ for all $\Bx, \By \in \J$.

Suppose that $P$ is in addition a subsemigroup of a discrete group $G$.
Then $\ca_\la(P)$ admits a normal coaction of $G$ such that
\[
[\ca_\la(P)]_g = \ol{\spn}\{V_{p_1}^* V_{q_1} \cdots V_{p_n}^* V_{p_n} \mid n \in \bN; p_1, q_1, \dots, p_n, q_n \in P; p_1^{-1} q_1\cdots p_n^{-1}q_n = g\}.
\]
Indeed, if $U \colon \ell^2(P) \otimes \ell^2(G) \to \ell^2(P) \otimes \ell^2(G)$ is the unitary operator given by
\[
U(\de_s \otimes \de_g) = \de_s \otimes \de_{sg} \foral s \in P, g \in G,
\]
then a routine calculation shows that the $*$-homomorphism
\[
\ca_\la(P) \stackrel{\simeq}{\longrightarrow}
\ca(V_p \otimes I \mid p \in P) \stackrel{\ad_{U}}{\longrightarrow}
\ca(V_p \otimes \la_p \mid p \in P)
\]
is a reduced coaction, and thus it lifts to a normal coaction on $\ca_\la(P)$.
The faithful conditional expectation on $\ca_\la(P)$ induced by $G$ coincides with the compression to the diagonal.
Lemma 3.1 of \cite{Li12} asserts that, if $p_1,q_1, \dots, p_n, q_n \in P$ are such that $p_1^{-1} q_1 \cdots p_n^{-1} q_n = e_G$ in $G$, then
\[
E_{[q_n^{-1} p_n \dots q_1^{-1} p_1 P]} = V_{p_1}^* V_{q_1} \cdots V_{p_n}^* V_{q_n}.
\] 
Note here that for such $p_i, q_i \in P$ we have
\[
V_{p_1}^* V_{q_1} \cdots V_{p_n}^* V_{q_n} e_r = e_r
\qiff
p_j^{-1} q_j \cdots p_n^{-1} q_n r \in P \foral j=1, \dots, n,
\]
otherwise it returns a zero value.
Since $p_1^{-1} q_1 \cdots p_n^{-1} q_n = e_G$ in $G$ it follows that
\[
q_{j-1}^{-1} p_{j-1} \cdots q_1^{-1} p_1 r = p_j^{-1} q_j \cdots p_n^{-1} q_n r \in P
\foral j=1, \dots, n.
\]
This shows that
\[
V_{p_1}^* V_{q_1} \cdots V_{p_n}^* V_{q_n} \neq 0
\qiff
q_n^{-1} p_n \dots q_1^{-1} p_1 P \neq \mt.
\]
Moreover, recall that
\[
q_n^{-1} p_n \dots q_1^{-1} p_1 p_1^{-1} q_1 \dots p_n^{-1} q_n P = (q_n^{-1} p_n \dots q_1^{-1} p_1 P) \cap P = q_n^{-1} p_n \dots q_1^{-1} p_1 P.
\]
Hence, without loss of generality, we will always assume that we write $\Bx = q_n^{-1} p_n \dots q_1^{-1} p_1 P$ with $p_1^{-1} q_1 \dots p_n^{-1} q_n = e_G$.
%
We can also use these relations to deduce computations for ideals in $\J$.
For example, let $p_1, q_1, \dots, p_n, q_n \in P$ with $p_1^{-1} q_1 \cdots p_n^{-1} q_n = e_G$ and $r_1, s_1, \dots, r_m, s_m \in P$ with $r_1^{-1} s_1 \cdots r_m^{-1} s_m = e_G$.
Then for the ideals
\[
\Bx := q_n^{-1} p_n \dots q_1^{-1} p_1 P, \;  \ \By := s_m^{-1} r_m \dots s_1^{-1} r_1 P \; \ \text{and} \; \ \Bz := s_m^{-1} r_m \dots s_1^{-1} r_1 q_n^{-1} p_n \dots q_1^{-1} p_1 P,
\]
we have that
\[
E_{[\Bz]} = V_{p_1}^* V_{q_1} \cdots V_{p_n}^* V_{q_n} V_{r_1}^* V_{s_1} \cdots V_{r_m}^* V_{s_m} = E_{[\Bx]} \cdot E_{[\By]} = E_{[\Bx \cap \By]}.
\]
We deduce that
\[
s_m^{-1} r_m \dots s_1^{-1} r_1 q_n^{-1} p_n \dots q_1^{-1} P 
= 
\Bz 
= 
\Bx \cap \By
=
s_m^{-1} r_m \dots s_1^{-1} r_1 P \cap p_n^{-1} q_n \dots p_1^{-1} q_1 P.
\]

\subsection{Gradings on C*-algebras}

A reduced coaction on a C*-algebra defines a topological grading over a discrete group $G$ in the sense of Exel \cite{Exe97}.
Recall that a \emph{topological grading} $\B := \{\B_g\}_{g \in G}$ of a C*-algebra consists of linearly independent subspaces that span a dense subspace of the C*-algebra they generate and are compatible with the structure of $G$, namely
\[
\B_g^* = \B_{g^{-1}} \text{ and } \B_g \cdot \B_h \subseteq \B_{g h} \foral g,h \in G.
\]
By \cite[Theorem 3.3]{Exe97} the independence condition can be substituted by the existence of a conditional expectation on $\B_e$.
A \emph{representation} $\Psi$ of $\B$ is a continuous linear map on $\bigcup_{g \in G} \B_g$ such that
\[
\Psi(b_g) \cdot \Psi(b_h) = \Psi(b_g \cdot b_h) \qand \Psi(b_g)^* = \Psi(b_g^*)
\]
for all $b_g \in \B_g$ and $b_h \in \B_h$.
We will say that $\Psi$ is \emph{$G$-equivariant} (or simply \emph{equivariant}) if $\ca(\Psi)$ admits a coaction such that
\[
\ca(\Psi) \longrightarrow \ca(\Psi) \otimes \ca_{\max}(G); \Psi(b_g) \mapsto \Psi(b_g) \otimes u_g.
\]
The maximal C*-algebra $\ca_{\max}(\B)$ of a topological grading $\B$ is the universal C*-algebra generated by $\B$ with respect to the representations of $\B$.
Due to universality, the maximal representation of $\B$ is equivariant.
The reduced C*-algebra $\ca_\la(\B)$ of $\B$ is defined by the left regular representation of $\B$ on $\ell^2(\B)$ given by
\[
\la(b_g) b_h = b_{g} b_{h} \foral g,h \in G.
\]
It follows that the left regular representation is equivariant with a normal coaction of $G$.
Indeed, for the unitary operator
\[
U \colon \ell^2(\B) \otimes \ell^2(G) \longrightarrow \ell^2(\B) \otimes \ell^2(G); b_g \otimes \de_h \mapsto b_g \otimes \de_{gh},
\]
we have that
\[
U \cdot ( \la(b_g) \otimes I ) = (\la(b_g) \otimes \la_g) \cdot U \foral g \in G.
\]
Hence we deduce the faithful $*$-homomorphism
\[
\ca_\la(\B) \stackrel{\simeq}{\longrightarrow} \ca_\la(\B) \otimes I \stackrel{\ad_U}{\longrightarrow} \ca_\la(\B) \otimes \ca_\la(G); \la(b_g) \mapsto \la(b_g) \otimes \la_g.
\]

A topological grading defines a \emph{Fell bundle}, and once a representation of a Fell bundle is established the two notions are the same.
In a loose sense, a Fell bundle $\B$ over a discrete group $G$ is a collection of Banach spaces $\{\B_g\}_{g \in G}$, often called its fibers, that satisfy canonical algebraic and C*-norm properties; for full details we see that \cite[Definition 16.1]{Exe17}.
So we will alternate between these two notions.

If $\de \colon \C \to \C \otimes \ca_{\max}(G)$ is a coaction on a C*-algebra $\C$, then we write $\{[\C]_g\}_{g \in G}$ for the induced topological grading.
It follows that $\de$ is normal if and only if the conditional expectation induced by $G$ is faithful.
An ideal $\I \lhd \C$ is called \emph{induced} if $\I = \sca{\I \cap [\C]_e}$.
If $\I \lhd C$ is an induced ideal, then $\de$ descends to a coaction on $C/\I$.

\begin{remark}\label{R:Exel}
We will be using the following facts from \cite{Exe97, Exe17}.
Let $\Psi$ be an equivariant representation of $\ca_{\max}(\B)$ and let 
\[
\Psi(\B) := \{\Psi(\B_g)\}_{g \in G}
\]
be the induced Fell bundle.
By \cite[Proposition 21.3]{Exe17} there are equivariant $*$-epimorphisms, making the following diagram
\[
\xymatrix{
\ca_{\max}(\B) \ar[rr] \ar[d] & & \ca_\la(\B) \ar[d] \\
\ca_{\max}(\Psi(\B)) \ar[r] & \ca(\Psi) \ar[r] & \ca_\la(\Psi(\B))
}
\]
commutative.
If $\Psi$ is an injective representation of $\B$, then $\ca_\la(\B) \simeq \ca_\la(\Psi(\B))$.
In that case the coaction on $\ca(\Psi)$ is normal if and only if $\ca(\Psi) \simeq \ca_\la(\B)$ by an equivariant $*$-isomorphism.
\end{remark}

\begin{remark}\label{R:induced}
Let $\C$ be a C*-algebra with a coaction $\de$ and let $\I = \sca{\I \cap [\C]_e}$ be an induced ideal of $\C$ so that $\C/\I$ admits a coaction.
A standard argument shows that, if $\Psi$ is an equivariant representation of $\C$, then there exists a commutative diagram
\[
\xymatrix{
\C \ar[d]^{q_{\I}} \ar[rr]^{\Psi} & & \ca(\Psi) \ar[d]^{q_{\Psi(\I)}} \\
\quo{\C}{\I} \ar[rr]^{\dot{\Psi}} & & \quo{\ca(\Psi)}{\Psi(\I)}
}
\]
so that $\dot{\Psi}$ is also equivariant.
Indeed, to this end it suffices to show that $\Psi(\I)$ is induced, i.e.,
\[
\Psi( \I \cap [\C]_e ) = \Psi(\I) \cap [\ca(\Psi)]_e.
\]
The inclusion ``$\subseteq$'' is immediate due to equivariance of $\Psi$.
For the reverse inclusion let the conditional expectation $E$, $E_{q_\I}$ and $E_\Psi$ on $\C$, $\C/\I$ and $\ca(\Psi)$, respectively.
For an element $x \in \Psi(\I) \cap [\ca(\Psi)]_e$ let $y \in \I$, so that $x = \Psi(y)$.
Then $q_\I E(y) = E_{q_\I} q_{\I}(y) = 0$, and so $E(y) \in \I \cap [\C]_e$.
Hence, we have that $x = E_\Psi(x) = E_\Psi \Psi(y) = \Psi E(y)$, and so $x \in \Psi( \I \cap [\C]_e )$.
\end{remark}

\section{Product systems and their operator algebras}

\subsection{Product systems}

Hilbert modules over C*-algebras are by now well understood.
The reader is addressed to \cite{Lan95,MT05} for an excellent introduction to the subject.

A \emph{C*-correspondence} $X$ over $A$ is a right Hilbert module over $A$ with a left action given by a $*$-homomorphism $\vphi_X \colon A \to \L X$, where $\L X$ denotes the C*-algebra of adjointable operators on $X$.
We write $\K X$ for the closed linear span of the rank one adjointable operators $\theta_{\xi, \eta}$.
For two C*-corresponden\-ces $X, Y$ over the same $A$ we write $X \otimes_A Y$ for the balanced tensor product over $A$.
We say that $X$ is \emph{unitarily equivalent} to $Y$ (symb. $X \simeq Y$) if there is a surjective adjointable operator $U \in \L(X,Y)$ such that $\sca{U \xi, U \eta} = \sca{\xi, \eta}$ and $U (a \xi b) = a U(\xi) b$ for all $\xi, \eta \in X$ and $a,b \in A$.
A representation of a C*-correspondence is a left and right module map on ${}_A X_A$ such that $t(\xi)^* t(\eta) = t(\sca{\xi,\eta})$ for all $\xi, \eta \in X$.
Every representation $t$ defines a $*$-representation on $\K X$, which we will denote again by $t$, such that $t(\theta_{\xi, \eta}) = t(\xi) t(\eta)^*$ for all $\xi, \eta \in X$.

\begin{definition}
Let $P$ be a unital discrete left-cancellative semigroup.
A \emph{concrete product system $X$ over $P$} is a family $\{X_p \mid p \in P\}$ of closed subspaces of some $\B(H)$ such that:
\begin{enumerate}
\item $A := X_{e}$ is a C*-algebra.
\item $X_p \cdot X_q \subseteq X_{pq}$ for all $p, q \in P$.
\item $X_p^* \cdot X_{pq} \subseteq X_q$ for all $p, q \in P$.
\end{enumerate}
\end{definition}

Since $P$ is left-cancellative we have that $q \in P$ in item (iii) above is unique (and thus this property is well-defined).
The following proposition shows that a concrete product system is a semigroup representation by C*-correspondences.
In particular, it asserts that every $X_{pq}$ is a right Hilbert module over $A$ and a left module over $[X_p \cdot X_p^*]$.

\begin{proposition}\label{P:equivalence}
Let $P$ be a unital discrete left-cancellative semigroup.
A family of closed subspaces $\{X_p\}_{p \in P}$ in a $\B(H)$ is a concrete product system over $P$, if and only if:
\begin{enumerate}
\item[\textup{(1)}] $A := X_e$ is a C*-algebra.
\item[\textup{(2)}] Every $X_{p}$ is a C*-correspondence over $A$ with the structure inherited by $\B(H)$.
\item[\textup{(3)}] $X_p \cdot X_q \subseteq X_{pq}$ for all $p, q \in P$.
\item[\textup{(4)}] $X_p \cdot X_p^* \cdot X_{pq} \subseteq [X_p \cdot X_q]$ for all $p, q \in P$.
\end{enumerate}
If items \textup{(1)}, \textup{(2)} and \textup{(3)} hold, then item \textup{(4)} is equivalent to each of the items below:
\begin{enumerate}
\item[\textup{(4')}] $[X_p \cdot X_p^* \cdot X_{pq}] = [X_p \cdot X_q]$ for all $p, q \in P$.
\item[\textup{(4'')}] $[X_p^* \cdot X_{pq}] = [X_p^* \cdot X_p \cdot X_q]$ for all $p, q \in P$.
\end{enumerate}
\end{proposition}

\begin{proof}
Suppose that $\{X_p\}_{p \in P}$ is a product system over $P$.
By applying property (ii) of the definition for $q \to e_P$ we get that $X_p$ is a right module over $A = X_e$.
By applying property (iii) of the definition for $q \to e_P$ we have that $X_p$ is a right Hilbert module over $A = X_e$ with the inner product space structure as inherited from $\B(H)$.
Recall here that the trivial Hilbert module structure on $\B(H)$ defines the same norm with the operator norm on $H$ due to the C*-identity.
Applying again property (ii) of the definition for $p \to e_P$ and $q \to p$ we get that $X_p$ is a left module over $A = X_e$.
As the operations take place inside $\B(H)$ which is a C*-correspondence over itself, we get that the left action is given by adjointable operators, and thus each $X_p$ is a C*-correspondence over $A = X_e$.
Finally, we have that 
\[
X_p \cdot (X_p^* \cdot X_{pq}) \subseteq X_p \cdot X_q \subseteq [X_p \cdot X_q],
\]
for all $p, q \in P$ by property (iii) of the definition.

Conversely, suppose that the family $\{X_p\}_{p \in P}$ satisfies items (1)-(4) of the statement.
Since $X_p$ is a C*-correspondence over $A = X_e$, then $X_p^* X_p X_p^*$ is dense in $X_p^*$, and thus
\[
X_p^* \cdot X_{pq} \subseteq [X_p^* \cdot X_p X_p^* X_{pq}] \subseteq 
[X_p^* \cdot X_p \cdot X_q] \subseteq [X_e \cdot X_q] \subseteq X_q.
\]

To finish the proof, suppose that items (1)-(3) of the statement hold.
Assume that item (4) holds.
Since $X_p X_p^* X_p$ is dense in $X_p$, then the equation 
\[
[X_p X_q] = [X_p X_p^* \cdot X_p X_q] \subseteq [X_p X_p^* X_{pq}] \subseteq [X_p X_q]
\]
gives item (4').
Assume now that item (4') holds.
By multiplying with $X_p^*$ and since $X_p^* X_p X_p^*$ is dense in $X_p^*$ we have that item (4') implies item (4'').
Likewise, multiplying with $X_p$ on both sides of item (4'') gives item (4), and the proof is complete.
\end{proof}

\begin{remark}
Fibres over invertible elements are compatible with adjoints.
That is, if $r$ is invertible in $P$ with inverse $r^{-1}$, then $[A \cdot X_r]^* = [A \cdot X_{r^{-1}}]$.
Indeed, by property (iii) of the definition we have that
\[
[A \cdot X_r]^* = [A \cdot X_r \cdot A]^* = [A \cdot X_r^* \cdot X_{r r^{-1}}] \subseteq [A \cdot X_{r^{-1}}].
\]
Taking adjoints and applying for $r^{-1}$ yields the symmetrical $[A \cdot X_{r^{-1}}] \subseteq [A \cdot X_r]^*$.

As a consequence, if $X = \{X_g\}_{g \in G}$ is a concrete product system over a group $G$ such that $X_e$ acts non-degenerately on every fibre then $X$ defines a Fell bundle, since taking adjoints satisfies
\[
X_g^* = [A \cdot X_g]^* = [A \cdot X_{g^{-1}}] = X_{g^{-1}} \foral g \in G.
\]
Conversely, if $\B = \{\B_g\}_{g \in G}$ is a concrete Fell bundle then it defines a concrete product system over $G$, with $\B_e$ acting non-degenerately on every fibre. 
Indeed, we have that
\[
\B_g \cdot \B_h \subseteq \B_{gh},
\qand
(\B_g)^* \cdot \B_{gh} = \B_{g^{-1}} \cdot \B_{gh} \subseteq \B_{h},
\]
while by definition $\B_e$ is a C*-algebra and non-degeneracy follows for example by \cite[Lemma 16.9]{Exe17}.
Therefore, product systems are in some sense a semigroup analogue of (concrete) Fell bundles.
\end{remark}

\begin{remark}[Fock representation]
The properties of a concrete product system are enough to provide a Fock space representation.
Consider $\B(H)$ with its trivial C*-correspondence structure.
We will be writing 
\[
\sca{\cdot, \cdot} \colon X_p \times X_p \longrightarrow A ; (\xi_p, \eta_p) \mapsto \sca{\xi_p, \eta_p} := \xi_p^* \cdot \eta_p,
\]
for the inner product induced on the $X_p$.
For every $p, q \in P$ we define the multiplication operators
\[
M_{\xi_p}^{q, pq} \colon X_q \longrightarrow X_{pq} ; \eta_q \mapsto \xi_p \cdot \eta_q,
\]
with the multiplication taken in $\B(H)$ where $X$ sits in.
It is clear that the operator norm of $M_{\xi_p}^{q, pq}$ is bounded by $\nor{\xi_p}_{X_p}$.
Associativity of the product gives that $M_{\xi_p}^{q, pq} \in \L(X_p, X_{pq})$ with
\[
(M_{\xi_p}^{q, pq} )^* = M_{\xi_p^*}^{pq, q}
\; \text{ where } \;
M_{\xi_p^*}^{pq, q}(\eta_{pq}) := \xi_p^* \cdot \eta_{pq} \in X_q.
\]
Indeed, a direct computation gives that
\[
\sca{\eta_{q}, M_{\xi^*_p}^{pq, q}(\eta_{pq})} = \eta_q^* \xi_p^* \eta_{pq} = \sca{M_{\xi_p}^{q, pq} \eta_q, \eta_{pq}}
\]
for all $\eta_q \in X_q$ and $\eta_{pq} \in X_{pq}$.
Consider the Fock space $\F X := \sum^\oplus_{r \in P} X_r$ which is a right Hilbert module over $A$.
For every $\xi_p \in X_p$ let
\[
\la(\xi_p) := \sumoplus_{r \in P} M_{\xi_p}^{r, pr} \; \text{ so that } \; \la(\xi_p)^* = \sumoplus_{r \in P} M_{\xi_p^*}^{pr, r},
\]
where the sum is taken in the s*-topology, and with the understanding that $\L(X_p, X_{pq}) \hookrightarrow \L(\F X)$ as the $(pq,p)$-entry.
By applying on $\eta_r \in X_r \subseteq \F X$ we derive that
\[
\la(\xi_p) \eta_r = \xi_p \cdot \eta_r
\qand
\la(\xi_p)^* \eta_r =
\begin{cases}
\xi_p^* \cdot \eta_r & \text{if } r \in pP, \\
0 & \text{if } r \notin pP.
\end{cases}
\]
Equivalently, let the Hilbert submodule $\sum_r^\oplus X_r \otimes e_r$ of $\B(H) \otimes \ell^2(P)$.
We can then write
\[
\la(\xi_p) = w ( M_{\xi_p} \otimes V_p ) w^* \foral p \in P,
\]
for the unitary $w \colon \F X \to \sum_r^\oplus X_r \otimes e_r$ with $w \xi_r := \xi_r \otimes e_r$, and the operators $M_x \in \L( \B(H) )$ with $M_x(y) := x y$ for all $x, y \in \B(H)$.
Consequently, we have that
\[
\la(\xi_p) \la(\xi_{q}) = \la(\xi_p \cdot \xi_{q}) 
\qand
\la(\xi_p)^* \la(\xi_{pq}) = \la(\xi_p^* \cdot \xi_{pq}).
\]
We write $\la \colon X \to \L(\F X)$ for this map and we refer to $\la$ as the \emph{Fock representation of $X$}.
\end{remark}

An equivalent categorical definition is available.
The key here is to consider a left module map of the left divisors on each fibre simultaneously.

\begin{definition}
Let $P$ be a unital discrete left-cancellative semigroup.
We say that a triple $\big( \{X_p\}_{p \in P}, \{i_{p}^{pq}\}_{p,q \in P}, \{u_{p,q}\}_{p,q \in P} \big)$ forms an \emph{abstract} product system $X$ if:
\begin{enumerate}
\item The space $X_p$ is a C*-correspondence over $A := X_e$ with a left action $\vphi_p$.
\item The space $X_{pq}$ (and thus, by associative extension, the C*-algebra $\K X_{pq} \simeq X_{pq} \otimes_A X_{pq}^*$) is a $\K X_p$-$A$-correspondence by a left action $i_{p}^{pq}$, where $i_e^p = \vphi_p$ and $i_p^p = \id_{\K X_p}$.
\item There are unitary maps
\[
u_{pq, r} \colon X_{pq} \otimes_A X_r \stackrel{\simeq}{\longrightarrow} [ i_{pq}^{pqr}(\K X_{pq}) X_{pqr}] \subseteq X_{pqr}
\]
which are $\K X_{p}$-$A$-adjointable, where $X_{pq} \otimes_A X_r$ is a $\K X_{p}$-$A$-correspondence by the left action $i_{p}^{pq} \otimes \id_{X_r}$; namely
\[
\iota_{p}^{pqr}(S) u_{pq, r} = u_{pq, r}( i_{p}^{pq}(S) \otimes \id_{X_r} ) \foral S \in \K X_p.
\]
\item The maps $\{u_{p,q}\}_{p,q \in P}$ are associative in the sense that the following diagram
\[
\xymatrix@R=2cm@C=2cm{
X_p \otimes_A X_q \otimes_A X_r \ar[rr]^{\id_{X_p} \otimes u_{q,r}} \ar[d]_{u_{p,q} \otimes \id_{X_r}} & & X_p \otimes_A X_{qr} \ar[d]^{u_{p, qr}} \\
X_{pq} \otimes_A X_r \ar[rr]^{u_{pq, r}} & & X_{pqr}
}
\]
is commutative for all $p, q, r \in P$.
\end{enumerate}
\end{definition}

In particular, we have that
\begin{equation} \label{eq:for}
i_p^{pq}(S) u_{p, q} = u_{p, q} (i_p^p(S) \otimes \id_{X_{q}}) = u_{p, q} (S \otimes \id_{X_{q}})
\foral 
S \in \K X_p.
\end{equation}
The left action of $\K X_{p}$ on $X_{p} \otimes_A X_{q}$ and on $[i_{p}^{pq}(\K X_{p}) X_{pq}]$ is non-degenerate and thus equation (\ref{eq:for}) extends to $\L X_{p}$, i.e.,
\begin{equation} \label{eq:for2}
i_p^{pq}(T) u_{p, q} = u_{p, q} (T \otimes \id_{X_{q}})
\foral 
T \in \L X_p.
\end{equation}
The compatibility relations then yield
\begin{equation} \label{eq:comp}
i_{p}^{pqr}(S) u_{p, qr} (\id _{X_p} \otimes u_{q,r}) = i_{pq}^{pqr} i_{p}^{pq}(S) u_{pq,r} (u_{p,q} \otimes \id_{X_r}) \foral S \in \K X_p.
\end{equation}
Indeed for the left hand side we apply equation (\ref{eq:for}) for $q$ in place of $qr$ and $S \in \K X_p$ to obtain
\begin{align*}
i_{p}^{pqr}(S) u_{p, qr} (\id_{X_p} \otimes u_{q,r})
& =
u_{p, qr} (S \otimes u_{q, r}) \\
& =
u_{p, qr} (\id_{X_p} \otimes u_{q, r}) (S \otimes \id_{X_q} \otimes \id_{X_r} ) \\
& =
u_{pq, r} (u_{p,q} \otimes \id_{X_r}) (S \otimes \id_{X_q} \otimes \id_{X_r}),
\end{align*}
where, in the last line, we used the compatibility of the $u$ maps from item (iv) of the definition of the abstract product system.
For the right hand side we apply the extended formula (\ref{eq:for2}) for $pq$ in place of $p$ and $r$ in place of $q$, and for $T = i_{p}^{pq}(S) \in \L X_{pq}$ so that
\begin{align*}
i_{pq}^{pqr} \left( i_{p}^{pq}(S) \right) u_{pq,r} (u_{p,q} \otimes \id_{X_r})
& =
u_{pq, r} (i_{p}^{pq}(S) \otimes \id_{X_r}) (u_{p,q} \otimes \id_{X_r}) \\
& =
u_{pq, r}(i_{p}^{pq}(S) u_{p,q} \otimes \id_{X_r}) \\
& =
u_{pq, r} (u_{p,q} \otimes \id_{X_r}) (S \otimes \id_{X_q} \otimes \id_{X_r}),
\end{align*}
where, in the last line, we applied once more equation (\ref{eq:for}).
This proves equation (\ref{eq:comp}).
Note here that if the maps $u_{p,q}$ are unitaries onto $X_{pq}$ then equation (\ref{eq:comp}) gives the familiar compatibility relation $i_{p}^{pqr} = i_{pq}^{pqr} i_{p}^{pq}$.

\begin{example}
The definition of abstract product systems we use here recovers the one by Fowler \cite[Definition 2.1]{Fow02}.
Let $P$ be a unital discrete left-cancellative semigroup.
Suppose that $X = \{X_p\}_{p \in P}$ is a family of C*-correspondences over the same C*-algebra $A$ such that:
\begin{enumerate}
\item $X_e = A$.
\item There are unitaries $X_p \otimes_A X_q \simeq_{u_{p,q}} X_{pq}$ for every $p, q \in P \setminus \{e\}$.
\item There are unitaries $A \otimes_A X_p \simeq_{u_{e, p}} \ol{A \cdot X_p}$ and $X_p \otimes_A A \simeq_{u_{p, e}} \ol{X_p \cdot A} = X_p$ for all $p \in P$.
\item The family of unitaries is associative in the sense that
\[
u_{pq, r} (u_{p,q} \otimes \id_{X_r}) = u_{p, qr} (\id_{X_p} \otimes u_{q,r}) \foral p,q,r \in P.
\]
\end{enumerate}
To make a distinction we will refer to such a family as a \emph{Fowler product system}.
The property $X_{pq} \simeq X_p \otimes_A X_q$ makes $X_{pq}$ a $\K X_p$-$A$-correspondence by $i_p^{pq}(S) := u_{p, q}( S \otimes \id_{X_q} ) u_{p, q}^*$.
Those families trivially preserve this structure, and thus a Fowler product system forms an abstract product system in our sense.
\end{example}

\begin{remark}\label{R:abstract}
We claim that an abstract product system $X$ attains an isometric representation in a C*-algebra as a concrete product system.
This will be achieved through a Fock space construction.
First we recall that every $\xi_p \in X_p$ defines an adjointable operator 
\[
\xi_p^* \colon X_p \otimes_A X_q \longrightarrow X_q ; \eta_p \otimes \eta_q \mapsto \vphi_{q}(\sca{\xi_p, \eta_p}) \eta_q.
\]
For every $\xi_p \in X_p$ we define the left creation operator
\[
\tau_{q}^{pq}(\xi_p) \colon X_q \longrightarrow X_{pq} ; \eta_q \mapsto u_{p,q}(\xi_p \otimes \eta_q).
\]
In particular, for $\xi_p, \eta_p, \zeta_p \in X_p$ we can use that the map $u_{p,q}$ is $\K X_p$-adjointable and get
\[
\tau_{q}^{pq}(\xi_p \sca{\eta_p, \zeta_p}) \eta_q 
=
u_{p,q}(\xi_p \sca{\eta_p, \zeta_p} \otimes \eta_q)
=
i_{p}^{pq}(\theta_{\xi_p, \eta_p}) u_{p,q}( \zeta_p \otimes \eta_q ).
\]
This map is $A$-linear and direct computations give that it is in $\L(X_q, X_{pq})$ with
\[
\tau_q^{pq}(\xi_p \sca{\eta_p, \zeta_p})^* \eta_{pq} = \zeta_p^* \left( u_{p,q}^*( i_{p}^{pq}(\theta_{\xi_p, \eta_p})^* \eta_{pq} ) \right).
\]
Now the association
\[
\tau_q^{pq} \colon X_p \sca{X_p, X_p} \longrightarrow \L(X_q, X_{pq})
\]
is $A$-linear and contractive, and thus it extends to the closure of $X_p \sca{X_p, X_p}$ which is $X_p$.
Therefore, we conclude that $\tau_q^{pq}(\xi_p)$ is an adjointable operator in $\L(X_q, X_{pq})$.

We can then form the full Fock space $\F X := \sum^\oplus_{r \in P} X_r$ and define the Fock representation 
\[
\la(\xi_p) := \sum_{r \in P} \tau_{r}^{pr}(\xi_p),
\]
where the sum is taken in the s*-topology.
We show that the family $\{\la(X_p)\}_{p \in P}$ satisfies the items of Proposition \ref{P:equivalence}.
Items (1) and (2) of Proposition \ref{P:equivalence} are straightforward.
Associativity of the maps $\{u_{p,q}\}_{p,q \in P}$ yields 
\[
\la(X_p) \la(X_q) \subseteq \la( u_{p,q}(X_p \otimes_A X_q) ) \subseteq \la(X_{pq}),
\]
which is item (3) of Proposition \ref{P:equivalence}.
Indeed, a direct computation for $\xi_p \in X_p$, $\xi_q \in X_q$ and $\eta_r \in X_r$ gives that
\begin{align*}
\la(\xi_p) \la(\xi_q) \eta_r
& =
u_{p, qr} (\id_{X_p} \otimes u_{q,r})( \xi_p \otimes \xi_q \otimes \eta_r ) \\
& =
u_{pq, r} (u_{p,q} \otimes \id_{X_r})( \xi_p \otimes \xi_q \otimes \eta_r )
=
\la( u_{p,q}(\xi_p \otimes \xi_q) ) \eta_r.
\end{align*}
Moreover, every map $u_{pq, r}$ is $\K X_p$-$A$-adjointable, and thus we get that
\begin{align*}
\la(X_p) \la(X_p)^* \la(X_{pq}) 
\subseteq \la(\K X_p \cdot X_{pq}) 
\subseteq \la(u_{p,q}(X_p \otimes_A X_q)) 
= [\la(X_p) \la(X_q)],
\end{align*}
which is item (4) of Proposition \ref{P:equivalence}.
Indeed, first we show that
\[
\la(x_p) \la(y_p)^* = \sumoplus_{r \in P} i_{p}^{pr}( \theta_{x_p, y_p} ) \equiv \la( \theta_{x_p, y_p} ) \in \L \left( \sumoplus_{r \in P} X_r \right).
\]
Since $X_p X_p^* X_p$ is dense in $X_p$ it suffices to show this for $y_p := \xi_p \sca{\eta_p, \zeta_p}$.
One one hand it is clear that
\[
\la(x_p) \la(y_p)^* \eta_r = 0 \textup{ whenever } r \notin p P.
\]
On the other hand, fix $\eta_{pr}$ so that $i_p^{pr}(\theta_{\xi_p, \eta_p})^* \eta_{pr} \in u_{p,r}(X_p \otimes_A X_r)$.
Thus we can write
\[
u_{p,r}^*( i_{p}^{pr}(\theta_{\xi_p, \eta_p})^* \eta_{pr}) = \sum_j v_p^{(j)} \otimes w_r^{(j)} \in X_p \otimes_A X_r,
\]
where the summation symbol is a short for convergence with respect to finite subsets (rather than series convergence).
Then we get that
\begin{align*}
\la(x_p) \la(y_p)^* \eta_{pr}
& =
u_{p,r} \bigg(\id_{X_p} \otimes \zeta_p^*\bigg) \bigg( x_p \otimes u_{p,r}^*\big( i_{p}^{pr}(\theta_{\xi_p, \eta_p})^* \eta_{pr} \big) \bigg) \\
& =
\sum_j u_{p,r} \bigg(\id_{X_p} \otimes \zeta_p^*\bigg) \bigg( x_p \otimes (v_p^{(j)} \otimes w_r^{(j)}) \big) \bigg) \\
& =
\sum_j u_{p,r}\big( x_p \otimes \vphi_{r}(\sca{\zeta_p, v_p^{(j)}}) w_r^{(j)} \big) 
 =
\sum_j u_{p,r}\big( x_p \sca{\zeta_p, v_p^{(j)}} \otimes w_r^{(j)} \big) \\
& =
\sum_j u_{p,r} \big( (\theta_{x_p, \zeta_p} \otimes \id_{X_r})(v_p^{(j)} \otimes w_r^{(j)}) \big) 
 =
i_{p}^{pr}(\theta_{x_p, \zeta_p}) \big(u_{p,r}( \sum_j v_p^{(j)} \otimes w_r^{(j)} ) \big) \\
& =
i_{p}^{pr}(\theta_{x_p, \zeta_p}) i_{p}^{pr}(\theta_{\eta_p, \xi_p}) \eta_{pr} 
=
i_{p}^{pr}(\theta_{x_p, y_p}) \eta_{pr},
\end{align*}
where we used that $u_{p,r}$ is a unitary onto its range.
Now we can deduce that
\[
\la(x_p) \la(y_p)^* \la(z_{pq}) = \la\big( i_{p}^{pq}(\theta_{x_p, y_p}) (z_{pq}) \big)
\]
for every $x_p, y_p \in X_p$ and $z_{pq}$.
Indeed, for every $\eta_r \in X_r$ we directly compute
\begin{align*}
\la(x_p) \la(y_p)^* \la(z_{pq}) \eta_r
& =
\la(x_p) \la(y_p)^* u_{pq, r}(z_{pq} \otimes \eta_r) 
=
i_{p}^{pqr}(\theta_{x_p, y_p}) u_{pq, r}(z_{pq} \otimes \eta_r) \\
& =
u_{pq, r}\bigg( (i_{p}^{pq}(\theta_{x_p, y_p}) \otimes \id_{X_r}) (z_{pq} \otimes \eta_r) \bigg) \\
& =
u_{pq, r}\big( i_{p}^{pq}(\theta_{x_p, y_p})z_{pq} \otimes \eta_r \big) 
=
\la \big( i_{p}^{pq}(\theta_{x_p, y_p}) z_{pq} \big) \eta_r.
\end{align*}

Conversely every concrete product system $\{X_p\}_{p \in P}$ in some $\B(H)$ forms an abstract product system.
First recall that the stabilized tensor product between two C*-correspondences is independent of the faithful representation on a common Hilbert space.
Thus the maps
\[
u_{p,q} \colon X_p \otimes_A X_q \longrightarrow [X_p \cdot X_q] \subseteq \B(H)
\]
are unitaries onto $[X_p X_q]$, and the family $\{u_{p,q}\}_{p,q \in P}$ satisfies associativity.
This is compatible with the identification of $\K X_p$ with $X_p \otimes_A X_p^*$, and in turn with $[X_p \cdot X_p^*]$.
We then define the left actions $i_p^{pq}$ by multiplication, i.e.,
\[
i_p^{pq}(\theta_{\xi_p, \eta_p}) \ze_{pq} = \xi_p \eta_p^* \ze_{pq} \in X_p X_p^* X_{pq} \subseteq X_{pq},
\]
where we used property (4) from Proposition \ref{P:equivalence}.
It follows that every $u_{pq, r}$ is a right Hilbert $A$-module map and in particular
\[
u_{pq, r}( X_{pq} \otimes_A X_r) = [X_{pq} X_r] = [X_{pq} X_{pq}^* X_{pqr}] = [i_{pq}^{pqr}(\K X_{pq}) X_{pqr}],
\]
where we used property (4') from Proposition \ref{P:equivalence}.
Finally, we need to check the compatibility relation, and it suffices to do so for $S = \theta_{\xi_p, \eta_p}$.
A direct computation yields
\[
i_p^{pqr}(\theta_{\xi_p, \eta_p}) u_{pq, r}( \ze_{pq} \otimes \ze_r )
=
\xi_p \eta_p^* \ze_{pq} \ze_r
=
u_{pq, r}( i_p^{pq}(\theta_{\xi_p, \eta_p}) \otimes \id_{X_r} ) (\ze_{pq} \otimes \ze_r).
\]
Hence the triple $\big( \{X_p\}, \{i_p^{pq}\}, \{u_{p,q}\} \big)$ defines an abstract product system structure.
\end{remark}

\begin{remark}\label{R:f is f}
If $X$ is an abstract product system, then $X$ and $\la(X)$ have unitarily equivalent Fock representations.
Towards this end let $\La$ be the Fock space representation of the concrete product system $\la(X)$ on $\F \la(X)$.
By using that $\la$ is injective on each $X_p$ we can define an isometric map
\[
U \colon \F \la(X) \longrightarrow \F X; \la(\xi_p) \mapsto \xi_p,
\]
first on finite sums and then extending by continuity.
In particular, $U$ is a unitary so that $U \La(\la(\xi_p)) = \la(\xi_p) U$.
Hence $U$ induces a spatial $*$-isomorphism $\ca(\La) \simeq \ca(\la)$ that fixes $X$.
\end{remark}

\subsection{Representations of product systems}

We next define the operators associated to a product system $X$.
As we will see there is no difference in assuming that $X$ is concrete.

\begin{definition}\label{D:Toeplitz}
Let $P$ be a unital discrete left-cancellative semigroup and let $X$ be a concrete product system over $P$.
A \emph{Toeplitz representation $t = \{t_p\}_{p \in P}$ of $X$} consists of a family of linear maps $t_p$ of $X_p$ such that
\begin{enumerate}
\item $t_e$ is a $*$-representation of $A := X_e$.
\item $t_p(\xi_p) t_q(\xi_q) = t_{pq}(\xi_p \xi_q)$ for all $\xi_p \in X_p$ and $\xi_q \in X_q$.
\item $t_p(\xi_p)^* t_{pq}(\xi_{pq}) = t_q( \xi_p^* \xi_{pq})$ for all $\xi_p \in X_p$ and $\xi_{pq} \in X_{pq}$.
\end{enumerate}
A representation is called \emph{injective} if it is injective on $A$, and thus it is injective on each $X_p$.
The \emph{Toeplitz algebra $\T(X)$ of $X$} is the universal C*-algebra generated by $A$ and $X$ with respect to the representations of $X$.
\end{definition}

By definition every $t_p$ is a representation of the C*-correspondence $X_p$ over $A$.
The definition of a Toeplitz representation is transferable to the abstract setting.

\begin{definition}
Let $P$ be a unital discrete left-cancellative semigroup and let $X$ be an abstract product system.
A \emph{Toeplitz representation $t = \{t_p\}_{p \in P}$ of $X$} consists of a family of linear maps $t_p$ of $X_p$ such that:
\begin{enumerate}
\item $(t_e, t_p, t_e)$ is a representation of the C*-correspondence $X_p$ for every $p \in P$.
\item $(t_p, t_{pq}, t_e)$ is a $\K X_p$-$A$-bimodule map on $(\K X_p, X_{pq}, A)$ for every $p,q \in P$, where $t_p$ is the representation of $\K X_p$ induced by $t_p$.
\item $t_p(\xi_p) t_q(\xi_q) = t_{pq}(u_{p,q}(\xi_p \otimes \xi_q))$ for every $p,q \in P$.
\end{enumerate}
\end{definition}

\begin{remark}
Remark \ref{R:abstract} shows that the Fock representation is a Toeplitz representation of an abstract product system $X$.
Moreover, since $\la$ is injective we have that the Toeplitz representations of an abstract product system $X$ are in bijection with the Toeplitz representations of the concrete product system $\la(X)$.
Due to Remark \ref{R:f is f} this holds for representations that factor through the Fock space construction.
Since we study the representations of a product system, there is thus no actual distinction between abstract and concrete product systems.
Henceforth we will refer to an abstract or a concrete product system, simply as a \emph{product system}.
\end{remark}

To relax notation, we will simply write $t$ for any $t_p$.
If $t$ defines a representation of $X$, then we will write $t_* \colon \T(X) \to \B(H)$ for the induced $*$-representation on $\T(X)$ and
\[
\ca(t) \equiv \ca(t_*) := t_*(\T(X)) = \ca(t(\xi_p) \mid \xi_p \in X_p; p \in P).
\]

Henceforth we will assume that $P$ is a unital subsemigroup of a group $G$.
Suppose that $\T(X)$ is faithfully represented as a representation $\wt{t}$ of $X$.
By the universal property of $\T(X)$ there is a canonical $*$-homomorphism
\[
\wt{\de} \colon \T(X) \longrightarrow \T(X) \otimes \ca_{\max}(G) : \wt{t}(\xi_p) \longmapsto \wt{t}(\xi_p) \otimes u_p.
\]
Sehnem \cite[Lemma 2.2]{Seh18} shows that $\wt{\de}$ is a non-degenerate and faithful coaction of $\T(X)$ when $X$ is a non-degenerate Fowler product system, with each spectral space $\T(X)_g$ with $g \in G$ be given by the products
\[
\wt{t}(\xi_{p_1})^* \wt{t}(\xi_{q_1}) \cdots \wt{t}(\xi_{p_n})^* \wt{t}(\xi_{q_n}) \qfor p_1^{-1} q_1 \cdots p_n^{-1} q_n = g.
\]
In light of Remark \ref{R:nd cis} we can extend this result to the generality of our context.

\begin{proposition}\label{P:t coaction}
Let $P$ be a unital subsemigroup of a discrete group $G$ and let $X$ be a product system over $P$.
Let $\wt{t}_*$ be a faithful representation of $\T(X)$.
Then there is a coaction of $G$ on $\T(X)$ such that 
\[
\wt{\de} \colon \T(X) \longrightarrow \T(X) \otimes \ca_{\max}(G) ; \wt{t}(\xi_p) \longmapsto \wt{t}(\xi_p) \otimes u_{p}.
\]
Moreover, each spectral space $[\T(X)]_g$ with $g \in G$ is given by the products of the form
\[
\wt{t}(\xi_{p_1})^* \wt{t}(\xi_{q_1}) \cdots \wt{t}(\xi_{p_n})^* \wt{t}(\xi_{q_n}) \qfor p_1^{-1} q_1 \cdots p_n^{-1} q_n = g.
\]
where $\wt{t}(\xi_{p_i}) = I$ or $\wt{t}(\xi_{q_i}) = I$ whenever $p_i = e_G$ or $q_i = e_G$, and $g \neq e_G$.
\end{proposition}

\begin{proof}
By the universal property of $\T(X)$ we obtain a well-defined $*$-homomorphism $\wt{\de}$ with left inverse $\id \otimes \chi$.
By considering $\id \otimes E_g$ for the Fourier co-efficients $E_g \colon \ca_{\max}(G) \to \bC$ we derive that each subspace
\[
[\T(X)]_g = \{x \in \T(X) \mid \wt{\de}(x) = x \otimes u_g\}
\]
is densely spanned by the products of the statement.
Since those subspaces densely span $\T(X)$, we derive that $\wt{\de}$ is a coaction on $\T(X)$.
\end{proof}

Let $t$ be a representation of $X$.
For $\Bx \in \J$ we define the \emph{$\bo{K}$-core on $\Bx$} of $\ca(t)$ by
\begin{align*}
\bo{K}_{\Bx, t_\ast} := \ol{\spn} \big\{t(X_{p_1})^* t(X_{q_1}) \cdots t(X_{p_n})^* t(X_{q_n}) \mid p_1^{-1} q_1 \cdots & p_n^{-1} q_n = e_G; \\
& q_n^{-1} p_n \dots q_1^{-1} p_1 P = \Bx \big\}.
\end{align*}
We do not claim that $\bo{K}_{\mt, t_{\ast}} = (0)$.
For a finite $\cap$-closed $\F \subseteq \J$ we define the \emph{$\bo{B}$-core on $\F$} by
\[
\bo{B}_{\F, t_\ast} := \sum_{\Bx \in \F} \bo{K}_{\Bx, t_\ast}.
\]

\begin{proposition}\label{P:ind lim}
Let $P$ be a unital subsemigroup of a discrete group $G$ and let $X$ be a product system over $P$.
Then the following hold for a representation $t$ of $X$:
\begin{enumerate}
\item[\textup{(a)}] For every $\Bx \in \J$, the space $\bo{K}_{\Bx, t_\ast}$ is a closed C*-subalgebra of $\ca(t)$.

\item[\textup{(b)}]
If $\F \subseteq \J$ is a finite $\cap$-closed set with $\mt \in \F$, then $\bo{B}_{\F, t_\ast}$ is a C*-subalgebra of $\ca(t)$.

\item[\textup{(c)}]
If $t$ is equivariant, then 
\[
[\ca(t)]_e = \ol{ \bigcup \{ \bo{B}_{\F, t_\ast} \mid \F \subseteq \J \textup{ finite and $\cap$-closed}, \mt \in \F\} }
\]
as the inductive limit of C*-algebras.
\end{enumerate}
\end{proposition}

\begin{proof}
(a) First we show that $\bo{K}_{\Bx, t_\ast}$ is a selfadjoint space.
Towards this end, consider
\[
t(X_{p_1})^* t(X_{q_1}) \cdots t(X_{p_n})^* t(X_{q_n}) \subseteq \bo{K}_{\Bx, t_\ast},
\]
for $p_i, q_i \in P$ such that $p_1^{-1} q_1 \cdots p_n^{-1} q_n = e_G$ and $q_n^{-1} p_n \dots q_1^{-1} p_1 P = \Bx$.
We then have that
\begin{align*}
q_n^{-1} p_n \dots q_1^{-1} p_1 P 
& =
P \cap (q_n^{-1} \cdot P) \cap (q_n^{-1} \cdot p_n \cdot P) \cap \cdots \cap (q_n^{-1} \cdot p_n \cdots q_1^{-1} \cdot p_1 \cdot P) \\
& =
(p_1^{-1} \cdot q_1 \cdots p_n^{-1} \cdot q_n \cdot P) \cap (p_1^{-1} \cdot q_1 \cdots p_n^{-1} \cdot P) \cap \cdots \cap P 
=
p_1^{-1} q_1 \dots p_n^{-1} q_n P,
\end{align*}
where the multiplications take place in $G$.
From this we derive that
\[
\left( t(X_{p_1})^* t(X_{q_1}) \cdots t(X_{p_n})^* t(X_{q_n}) \right)^*
=
t(X_{q_n})^* t(X_{p_n}) \cdots t(X_{q_1})^* t(X_{p_1}) 
\subseteq 
\bo{K}_{\Bx, t_\ast}
\]
as well, and thus $\bo{K}_{\Bx, t_\ast}$ is selfadjoint.
Next recall that if we have
\[
\Bx = q_n^{-1} p_n \dots q_1^{-1} p_1 P \qand \By = s_m^{-1} r_m \dots s_1^{-1} r_1 P,
\]
with $p_1^{-1} q_1 \cdots p_n^{-1} q_n = e_G$ and $r_1^{-1} s_1 \cdots r_m^{-1} s_m = e_G$, then
\[
\Bx \cap \By = q_n^{-1} p_n \dots q_1^{-1} p_1 s_m^{-1} r_m \dots s_1^{-1} r_1 P,
\]
and of course $q_n^{-1} p_n \dots q_1^{-1} p_1 s_m^{-1} r_m \dots s_1^{-1} r_1 = e_G$.
Hence we obtain that
\[
\bo{K}_{\Bx, t_\ast} \cdot \bo{K}_{\By, t_\ast} \subseteq \bo{K}_{\Bx \cap \By, t_\ast}.
\]
From this we derive that every $\bo{K}$-core is an algebra (as $\Bx \cap \Bx = \Bx$).

\smallskip

\noindent
(b) Let an enumeration $\F = \{\Bx_0, \dots, \Bx_n\}$ that dominates the partial order of inclusions in $\F$, so that $\Bx_0 = \mt$.
For $k \leq n$ we write $\F_k := \{\Bx_0, \dots, \Bx_k\}$.
Due to the enumeration and since $\F$ is $\cap$-closed we have that every $\F_k$ is $\cap$-closed.
From part (a), every $\bo{K}_{\Bx_n, t_\ast}$ is a C*-algebra and
\[
\bo{K}_{\Bx_0, t_{\ast}} \cdot \bo{K}_{\Bx_1, t_\ast} \subseteq \bo{K}_{\Bx_0, t_{\ast}},
\]
since $\Bx_0 \cap\Bx_1 = \Bx_0$.
Hence we have that
\[
\bo{K}_{\Bx_0, t_\ast} \lhd \bo{B}_{\{\Bx_0, \Bx_1\}, t_\ast} = \bo{K}_{\Bx_0, t_{\ast}} + \bo{K}_{\Bx_1, t_{\ast}},
\]
and thus $\bo{B}_{\{\Bx_0, \Bx_1\}, t_\ast}$ is closed as the sum of an ideal with a C*-subalgebra.
Suppose now that we have shown that $\bo{B}_{\F_k, t_\ast}$ is closed for $\F_k := \{\Bx_0, \Bx_1, \dots, \Bx_k\}$ for $k < n$.
If $m \leq k$ then $\Bx_m \not\supseteq \bo{x}_{k+1}$, i.e., $\Bx_{m} \cap \Bx_{k+1} \neq \bo{x}_{k+1}$, and so $\Bx_{m} \cap \Bx_{k+1} \in \F_{k}$.
Hence we have that
\[
\bo{B}_{\F_k, t_\ast} \lhd \bo{B}_{\F_{k+1}, t_\ast} = \bo{B}_{\F_k, t_\ast} + \bo{K}_{\Bx_{k+1}, t_\ast},
\]
and thus $\bo{B}_{\F_{k+1}, t_\ast}$ is closed as the sum of an ideal with a C*-subalgebra.
Inducting on $k$ shows that $\bo{B}_{\F, t_\ast}$ is a C*-subalgebra.

\smallskip

\noindent
(c) By assumption we have that every non-zero element of the form
\[
t(\xi_{p_1})^* t(\xi_{q_1}) \cdots t(\xi_{p_n})^* t(\xi_{q_n}) \qfor p_1^{-1} q_1 \cdots p_n^{-1} q_n = e_G
\]
sits inside $\bo{K}_{\Bx, t_\ast}$ for $\Bx := q_n^{-1} p_n \dots q_1^{-1} p_1 P$.
Proposition \ref{P:t coaction} yields that $[\ca(t)]_e$ is thus densely spanned by the $\bo{K}$-cores.
Thus by appending the empty set and all intersections to any finite subset of $\J$ gives the inductive limit form.
\end{proof}

\subsection{Exel's Bessel inequality}\label{Ss:Bessel}

The Bessel inequality is an important tool in the Fell bundles theory; see \cite[Theorem 3.3]{Exe97} and/or \cite[Lemma 17.14]{Exe17}.
We will see here an application to C*-algebras of product systems.
Towards this end let $t$ be an equivariant injective representation of $X$ in some $\B(H)$ with conditional expectation $E_{t_\ast}$, and consider the induced Fell bundle
\[
\B_{t_\ast} = \{ [\ca(t)]_g \}_{g \in G}.
\]
Let $Y_{t_\ast}$ be the $\B_e$-Hilbert module with respect to the inner product structure
\[
\sca{\eta_1, \eta_2}_{Y_{t_\ast}} := E_{t_{\ast}}(\eta_1^* \eta_2) \foral \eta_1, \eta_2 \in \sum_g \B_{t_\ast, g}.
\]
By construction we have that $\sum_g \B_{t_\ast, g} \hookrightarrow Y_{t_\ast}$, and in turn in \cite[Theorem 3.3]{Exe97} it is shown that $Y_{t_\ast}$ is unitarily equivalent to $\ell^2(\B_{t_\ast})$ by the map $U \colon b_g \mapsto j(b_g)$, for the inclusion $j \colon \B_{t_\ast, g} \hookrightarrow \ell^2(\B_{t_\ast})$.
Moreover, in the proof of \cite[Theorem 3.3]{Exe97} it is established that
\[
\sca{c x, c x}_{Y_{t_\ast}} \leq \nor{c}^2_{\B(H)} \sca{x, x}_{Y_{t_\ast}} \foral c \in \ca(t), x \in \sum_g \B_{t_\ast, g}.
\]
Applying on $M_n(\ca(t))$ yields a matricial Bessel inequality on $\sum_{k=1}^n Y_{t_\ast}$ given by
\begin{equation*}
\sca{ [c^{ij}] [x^k], [c^{ij}] [x^k]}_{\sum_{k=1}^n Y_{t_\ast}} \leq \nor{[c^{ij}]}^2_{M_n(\B(H))} \sca{[x^k], [x^k]}_{\sum_{k=1}^n Y_{t_\ast}}
\end{equation*}
for $c^{ij} \in \ca(t)$ and $x^k \in \sum_g \B_{t_\ast, g}$, $i, j, k = 1, \dots, n$.

For a product system version, let $Y_t$ be the $A$-submodule of $Y_{t_\ast}$ densely spanned just by the $t(X_p)$.
Since $t$ is injective on $X$ we may see $\F X \hookrightarrow \ell^2(\B_{t_\ast})$ and thus $Y_t \simeq^U \F X$.
Let $i, j, k \in \{1, \dots, n\}$, and consider
\[
b^{ij} \in \T(X) \qand {\eta}^k := \sum_{q \in P} t(\eta^k_q) \quad \textup{such that} \quad t_\ast(b^{ik}) \eta^k \in \sum_p t(X_p).
\]
Then $[t_\ast(b^{ij})] [\eta^k] \in \sum_{k=1}^n Y_{t}$ and, by restriction, Bessel's inequality yields
\begin{align*}
\sca{ [ t_\ast(b^{ij}) ] [\eta^k], [ t_\ast(b^{ij}) ] [\eta^k] }_{\sum\limits_{k=1}^n Y_t}
& =
\sca{ [ t_\ast(b^{ij}) ] [\eta^k], [ t_\ast(b^{ij}) ] [\eta^k] }_{\sum\limits_{k} Y_{t_\ast}} \\
& \leq 
\nor{ [ t_\ast(b^{ij}) ] }^2_{M_n(\B(H))} \cdot \sca{ [\eta^k], [\eta^k]}_{\sum\limits_{k} Y_{t_\ast}} \\
& =
\nor{ [ t_\ast(b^{ij}) ] }^2_{M_n(\B(H))} \cdot \sca{ [\eta^k], [\eta^k]}_{\sum\limits_{k=1}^n Y_t}.
\end{align*}
In particular, we get the following corollary.

\begin{corollary}\label{C:Bessel}
Let $P$ be a unital subsemigroup of a discrete group $G$ and let $X$ be a product system over $P$.
Let $t \colon X \to \B(H)$ be an equivariant injective representation of $X$, and let $Y_t$ be the $A$-module as constructed above.
For $i,j,k \in \{1, \dots, n\}$ let $\eta^k \in Y_{t}$ and $b^{ij} \in \T(X)$.
If $t_\ast(b^{ik}) \eta^k \in Y_{t_\ast}$ for all $i,k$, then 
\[
\| [ t_\ast(b^{ij}) ] \cdot [\eta^k] \|_{\sum\limits_{k=1}^n Y_{t}} 
\leq 
\nor{ [ t_\ast(b^{ij}) ] }_{M_n(\B(H))} \cdot \| [\eta^k] \|_{\sum\limits_{k=1}^n Y_{t}}.
\]
\end{corollary}

In \cite[Proposition 3.1]{Seh21}, Sehnem produces a hands-on Bessel-type inequality for elements in $\T_\la(X)^+$ and representations of $\T_\la(X)$.
In short, let $\Phi$ be an equivariant $*$-representation of $\T_\la(X)$ that admits a conditional expectation $E$, and let $E_\la$ be the faithful conditional expectation on $\T_\la(X)$.
For $\xi := \sum_{p \in P} \xi_p$ in $\F X$ set $\la(\xi) = \sum_{p \in P} \la(\xi_p)$.
Then taking $\eta := \sum_{q \in P} \eta_q$ in $\F X$ yields 
\begin{align*}
E \Phi( \la( \la(\xi) \eta )^* \la( \la(\xi) \eta ) )
& =
E \Phi( \la(\eta)^* \la(\xi)^* \la(\xi) \la(\eta) )
\leq
\nor{\Phi \la(\xi)}^2 E \Phi( \la(\eta)^* \la(\eta) ).
\end{align*}
Using that $E \Phi = \Phi E_\la$ gives
\[
E \Phi( \la( \la(\xi) \eta )^* \la( \la(\xi) \eta ) )
=
\Phi E_\la ( \la( \la(\xi) \eta )^* \la( \la(\xi) \eta ) )
=
\Phi \la( \sum_{p, q} \sca{\xi_p \eta_q, \xi_p \eta_q} ),
\]
and that
\[
E \Phi( \la(\eta)^* \la(\eta) )
=
\Phi E_\la( \la(\eta)^* \la(\eta) )
=
\Phi \la( \sum_{q \in P} \sca{\eta_q, \eta_q} ).
\]
If $\Phi$ is injective on $A$, then passing to norms yields
\[
\nor{\la(\xi) \eta}_{\F X} \leq \nor{\Phi \la(\xi)} \cdot \nor{\eta}_{\F X}.
\]
Similar arguments apply to the matrix levels.
As a consequence equivariant representations of $\T_\la(X)$ that are injective on $A$ are completely isometric on $\T_\la(X)^+$ \cite[Corollary 3.5]{Seh21}.
Corollary \ref{C:Bessel} allows an extension of this result, so that $\T_\la(X)^+$ receives completely contractive representations from the algebra of any equivariant injective representation of $X$ (we will show this in Proposition \ref{P:cc hom cstar}).

\section{Covariant representations of product systems}

\subsection{Fock-covariant representations}

The Fock representation $\la$ defines an injective representation of $X$ and hence it induces a representation $\la_*$ of $\T(X)$.
It has played a crucial role in the identification of the co-universal object in \cite{DKKLL20, KKLL21}, and it is at the epicenter of the covariant relations.

\begin{definition}
Let $P$ be a unital discrete left-cancellative semigroup and let $X$ be a product system over $P$.
The \emph{Fock algebra} $\T_\la(X)$ is the C*-algebra generated by the Fock representation $\la_*$.
The \emph{Fock tensor algebra $\T_\la(X)^+$ of $X$} is the (possibly nonselfadjoint) subalgebra of $\T_\la(X)$ generated by the image of $X$ under $\la$.
\end{definition}

As observed in \cite{DKKLL20}, when $P$ is a unital subsemigroup of a group $G$, then the Fock algebra admits a normal coaction by $G$.
Indeed, a direct computation gives that 
\[
U \cdot (\la(\xi_p) \otimes I) = (\la(\xi_p) \otimes \la_p) \cdot U \foral p \in P,
\]
for the unitary $U \colon \F X \otimes \ell^2(G) \to \F X \otimes \ell^2(G)$ given by
\[
U (\xi_r \otimes \de_g) = \xi_r \otimes \de_{r g}
\foral
r \in P, g \in G.
\]
Recall here that by \cite[p.37]{Lan95} we have 
\[
\T_\la(X) \otimes \ca_\la(G) \subseteq \L(\F X) \otimes \B(\ell^2(G)) \subseteq \L( \F X \otimes \ell^2(G) ),
\]
where we see $\ell^2(G)$ as a right Hilbert module over $\bC$, and $\F X \otimes \ell^2(G)$ denotes the exterior product of Hilbert modules.
Hence the faithful $*$-homomorphism
\[
\T_\la(X) \stackrel{\simeq}{\longrightarrow}
\ca(\la(\xi_p) \otimes I \mid p \in P) \stackrel{\ad_{U}}{\longrightarrow}
\ca(\la(\xi_p) \otimes \la_p \mid p \in P)
\]
defines a reduced coaction on $\T_\la(X)$, and thus it lifts to a normal coaction $\de$ on $\T_\la(X)$.
We can see that $\la_* \colon \T(X) \to \T_\la(X)$ is equivariant with respect to the coactions of $G$.

The coaction of $G$ on $\T_\la(X)$ induces a grading and thus a Fell bundle.
We will use this as the defining property for \emph{Fock covariance}.

\begin{definition}
Let $P$ be a unital subsemigroup of a discrete group $G$ and let $X$ be a product system over $P$.
For $\J_{\cov, e}^{\fock} := \ker\la_\ast \cap [\T(X)]_e$ let the induced ideal
\[
\J_{\cov}^{\fock} := \sca{\ker\la_\ast \cap [\T(X)]_e} \lhd \T(X).
\]
We write $\T_{\cov}^{\fock}(X)$ for the equivariant quotient of $\T(X)$ by $\J_{\cov}^{\fock}$.
We define the \emph{Fock-covariant} bundle of $X$ be the Fell bundle
\[
\F\C_G X := \Big\{ [\T_{\cov}^{\fock}(X)]_g \Big\}_{g \in G},
\]
given by the coaction of $G$ on $\T_{\cov}^{\fock}(X)$.
A representation of $\F\C_G X$ will be called a \emph{Fock-covariant representation of $X$}.
\end{definition}

\begin{remark}
Since the bimodule properties are graded we have that every representation of $\F\C_G X$ is a representation of $X$. 
Moreover $\T_\la(X)$ is a representation of $\F\C_G X$ by definition, and it admits a normal coaction, thus $\T_\la(X) \simeq \ca_\la(\F\C_G X)$.
Hence we have the canonical $*$-epimorphisms
\[
\T(X) \longrightarrow \T_{\cov}^{\fock}(X) \longrightarrow \T_\la(X).
\]
We will use the notation $\iota \colon \F\C_G X \hookrightarrow \T_{\cov}^{\fock}(X) = \ca_{\max}(\F\C_G X)$ for the natural embedding of the Fell bundle in its universal C*-algebra.
Therefore, we have
\[
\iota(b_{\Bx}), \iota \la(\xi_r) \in \T_{\cov}^{\fock}(X) \foral b_{\Bx}, \la(\xi_r) \in \T_\la(X).
\]
If $t$ is a Fock-covariant representation of $X$ we thus have
\[
t_\ast \iota \la(\xi_r) = t(\xi_r) \foral \xi_r \in X_r,
\]
for the induced $t_\ast \colon \T_{\cov}^{\fock}(X) \to \ca(t)$.
\end{remark}

As we are about to see, Fock-covariant representations satisfy that $\bo{K}_{\mt, t_\ast} = (0)$.
Hence by Proposition \ref{P:ind lim} their fixed point algebras admit an inductive limit realization over finite $\cap$-closed subsets of $\J$ (discarding the empty set).

\begin{proposition}\label{P:Fock cov rel}
Let $P$ be a unital subsemigroup of a discrete group $G$ and let $X$ be a product system over $P$.
If $t$ is a Fock-covariant representation of $X$, then $\bo{K}_{\mt, t_\ast} = (0)$.
\end{proposition}

\begin{proof}
It suffices to show this for the Fock representation.
We will show that if $\bo{K}_{\Bx, \la_\ast} \neq (0)$ then $\Bx \neq \mt$.
Towards this end, let $p_i, q_i \in P$ such that $p_1^{-1} q_1 \cdots p_n^{-1} q_n = e_G$, and $\xi_{p_i}$, $\xi_{q_i}$ such that
\[
\la(\xi_{p_1})^* \la(\xi_{q_1}) \cdots \la(\xi_{p_n})^* \la(\xi_{q_n}) \neq 0.
\]
Then there is an $r \in P$ and a $\xi_r \in X_r$ such that
\[
X_r \ni \xi_{p_1}^* \xi_{q_1} \cdots \xi_{p_n}^* \xi_{q_n} \xi_r = \la(\xi_{p_1})^* \la(\xi_{q_1}) \cdots \la(\xi_{p_n})^* \la(\xi_{q_n}) \xi_r \neq 0.
\]
In particular, we get $p_{j+1}^{-1} q_{j+1} \cdots p_n^{-1} q_n r = q_j^{-1} p_{j} \cdots q_1^{-1} p_1 r \in P$ for every $j = 1, \dots, n$.
Therefore, the ideal $q_n^{-1} p_n \dots q_1^{-1} p_1 P$ contains $r$, and hence it is non-empty.
\end{proof}

We will be checking at every level that the constructions coming from $\T_\la(X)$ are independent of the group embedding.
This boils down to showing that their fixed point algebras are independent of the group embedding.
The first part of the proof of \cite[Lemma 3.9]{Seh18} can be transferred in our setting.

\begin{proposition}\label{P:la fpa same}
Let $P$ be a unital subsemigroup and let $X$ be a product system over $P$.
Suppose that $P$ admits two group embeddings $i_G \colon P \to G$ and $i_H \colon P \to H$.
Then 
\begin{align*}
& \ol{\spn} \big\{\la(X_{p_1})^* \la(X_{q_1}) \cdots \la(X_{p_n})^* \la(X_{q_n}) \mid i_G(p_1)^{-1} i_G(q_1) \cdots i_G(p_n)^{-1} i_G(q_n) = e_G; \\
& \hspace{11.5cm}
q_n^{-1} p_n \dots q_1^{-1} p_1 P = \Bx \big\}
= \\
& \hspace{.5cm} =
\ol{\spn} \big\{\la(X_{p_1})^* \la(X_{q_1}) \cdots \la(X_{p_n})^* \la(X_{q_n}) \mid i_H(p_1)^{-1} i_H(q_1) \cdots i_H(p_n)^{-1} i_H(q_n) = e_H; \\
& \hspace{11.5cm}
q_n^{-1} p_n \dots q_1^{-1} p_1 P = \Bx \big\},
\end{align*}
and thus the fixed point algebras of $\T_\la(X)$ coincide, i.e., 
\[
[\T_\la(X)]_{e_G} = [\T_\la(X)]_{e_H}.
\]
\end{proposition}

\begin{proof}
For convenience, let $G$ be the enveloping group of $P$ and let $\ga \colon G \to H$ be the canonical epimorphism that fixes $P$.
By definition we have a commutative diagram
\[
\xymatrix{
\T_\la(X) \ar[rr] \ar[drr] & & \T_\la(X) \otimes \ca_{\max}(G) \ar[d] \\
& & \T_\la(X) \otimes \ca_{\max}(H)
}
\]
and therefore $[\T_\la(X)]_{e_G} \subseteq [\T_\la(X)]_{e_H}$.
In particular, we have the inclusion ``$\subseteq$'' for the $\bo{K}$-cores of the statement.
For the reverse inclusion let a non-zero element
\[
0 \neq b_{\Bx} = \la(\xi_{p_1})^* \la(\xi_{q_1}) \cdots \la(\xi_{p_n})^* \la(\xi_{q_n}) \in [\T_\la(X)]_{e_H}
\]
for $\ga(p_1)^{-1} \ga(q_1) \cdots \ga(p_n)^{-1} \ga(q_n) = e_H$.
We will show that $b_{\Bx} \in [\T_\la(X)]_{e_G}$; equivalently that $p_1 q_1^{-1} \cdots p_n q_n^{-1} = e_G$.
To reach a contradiction suppose that $p_1 q_1^{-1} \cdots p_n q_n^{-1} \neq e_G$ and set
\[
X := q_1 p_2^{-1} \dots p_n^{-1} q_n P
\qand
Y := \ga^{-1}(\ga(q_1) \ga(p_2)^{-1} \dots \ga(p_n)^{-1} \ga(p_n) \ga(P)).
\]
We show that $X = Y = \mt$.
Then Proposition \ref{P:ind lim} yields
\[
\la(\xi_{q_1}) \la(\xi_{p_2})^* \cdots \la(\xi_{p_n})^* \la(\xi_{q_n}) = 0,
\]
and thus the contradiction $b_{\Bx} = 0$.
First note that since $\ga$ is one-to-one on $P$ and a group homomorphism we have that $X=Y$.
Secondly in order to show that $X = Y = \mt$, suppose that $X \neq \mt$ so that there are $r,s \in P$ such that $s = q_1 p_2^{-1} \cdots p_n^{-1} q_n r$ and thus $\ga(p_1 s) = \ga(r)$.
But as $\ga$ is one-to-one on $P$ we have that $p_1 s = r$ and thus $p_1^{-1} q_1 \cdots p_n^{-1} q_n r = p_1 s = r$ giving the contradiction $p_1^{-1} q_1 \cdots p_n^{-1} q_n = e_G$.
\end{proof}

As an immediate corollary, we derive that $\T_{\cov}^{\fock}(X)$ is independent of the group embedding, in the following sense.

\begin{corollary}
Let $P$ be a unital subsemigroup and let $X$ be a product system over $P$.
Suppose that $P$ admits two group embeddings $i_G \colon P \to G$ and $i_H \colon P \to H$.
Then there exists a $*$-isomorphism
\[
\ca_{\max}(\F\C_G X) \longrightarrow \ca_{\max}(\F\C_H X); \xi_{i_G(p)} \mapsto \xi_{i_H(p)}.
\]
In particular, there exists a Fell bundle isomorphism $\F\C_G X \to \F\C_H X$ such that
\[
[\T_{\cov}^{\fock}(X)]_{i_G(p_1)^{-1} i_G(q_1) \cdots i_{G}(p_n)^{-1} i_G(q_n)} \longrightarrow [\T_{\cov}^{\fock}(X)]_{i_H(p_1)^{-1} i_H(q_1) \cdots i_{H}(p_n)^{-1} i_H(q_n)},
\]
isometrically for every $p_1, q_1, \dots, p_n, q_n \in P$.
\end{corollary}

It transpires that the $*$-isomorphism passes down to the C*-algebras, namely
\[
\ca_{\la}(\F\C_G X) \simeq \ca_{\la}(\F\C_H X) \simeq \T_\la(X).
\]
In the generality that we work here, we see that $\T_\la(X)$ may not admit a Wick ordering, i.e., it may not be densely spanned by $\la(X_p) \la(X_q)^*$.
This happens for compactly aligned product systems over right LCM-semigroups \cite{DKKLL20, KL19b}.
Product systems over right LCM semigroups generalize product systems over quasi-lattices in the sense of Fowler \cite{Fow02}.

\begin{example}
Let $P$ be a right LCM semigroup in the sense of \cite{Law12}, i.e., $\J = \{ p P \mid p \in P\} \cup \{\mt\}$.
A Fowler's product system $X$ over $P$ is \emph{non-degenerate} and \emph{compactly aligned} if
\begin{enumerate}
\item there are associative multiplication rules so that $X_p \otimes_A X_q \stackrel{u_{p,q}}{\simeq} X_{pq}$ for all $p, q \in P$.
\item if $pP \cap qP = wP$, then $i_{p}^{w}(k_p) i_q^w(k_q) \in \K X_w$ for all $k_p \in \K X_p$ and $k_q \in \K X_q$.
\end{enumerate}
Here we use the notation
\[
i_p^{pq}(S) = u_{p, q}( S \otimes \id_{X_q} ) u_{p, q}^* \foral S \in \L X_p.
\]
Property (i) implies that the left action of $A$ on every $X_p$ is non-degenerate.
Property (ii) implies that the Fock representation satisfies the \emph{Nica-covariance} relation in the sense of \cite{KL19b}, i.e., for $k_p \in \K X_p$ and $k_q \in \K X_q$ we have
\[
\la(k_p) \la(k_q)
=
\begin{cases}
\la(i_{p}^w(k_p) i_q^w(k_q)) & \text{ if } pP \cap qP = wP, \\
0 & \text{ otherwise},
\end{cases}
\]
where we use $\la$ for the induced representation on each $\K X_r$.
Note that by \cite[Proposition 2.4]{DKKLL20} we have that $\la(i_p^{x p}(k_p)) = \la(k_p)$ for all $k_p \in \K X_p$ and $x \in P \cap P^{-1}$, and so the above relations do not depend on the choice of the LCM.
Nica-covariance implements a Wick ordering on $\T_\la(X)$, i.e., 
\[
\T_\la(X) = \ol{\spn}\{ \la(X_p) \la(X_q)^* \mid p,q \in P\}.
\]
Hence $\la(\xi_p) \la(\xi_q)^* \in [\T_\la(X)]_e$ if and only if $p=q$.
We see that
\[
\la(X_q) \la(X_q)^* \in \bo{K}_{pP, \la_*} \textup{ iff } pP = q q^{-1}P = qP.
\]
Then $q = p x$ for $x \in P \cap P^{-1}$, and therefore by \cite[Proposition 2.4]{DKKLL20} we get that
\[
\la(X_q) \la(X_q)^* \subseteq [\la(X_p) \la(X_x) \la(X_x)^* \la(X_p)^*] = [\la(X_p) \la(X_p)^*].
\]
Consequently, we derive that
\[
\bo{K}_{pP, \la_*} = [\la(X_p) \la(X_p)^*].
\]
It is worth noting that the compact alignment of $X$ is in fact equivalent with $\la$ satisfying the Nica-covariance relation.
This has been observed by Katsoulis in \cite[Proposition 3.2]{Kat20} but the proof applies in this more general context as well, see \cite[Proposition 4.7]{KKLL21b}.
The key point is that Nica-covariance and Fock-covariance coincide in this setting by \cite[Proposition 4.3]{DKKLL20}.
\end{example}

\begin{example}
Let $X_p = \bC$ for every $p \in P$.
In this case we have that $\bo{K}_{\Bx, \la_\ast} = \bC \cdot E_{[\Bx]}$ for every $\Bx \in \J$.
Laca--Sehnem \cite{LS21} located the Fock-covariant representations of this trivial product system.
They show that the Fock-covariant C*-algebra $\T_{\cov}^{\fock}(X)$ in this case is the universal C*-algebra generated by an isometric representation $\{v_p \mid p \in P\}$ of $P$ and projections $\{e_{\Bx} \mid \Bx \in \J\}$ with $e_\mt = 0$ such that:
\[
\prod_{\By \in \F} (e_{\Bx} - e_{\By}) = 0 \qfor \Bx = \bigcup_{\By \in \F} \By, \textup{ finite } \F \subseteq \J.
\]

Several semigroup C*-algebras were considered by Li \cite{Li12}, and in particular the one obtained by the \emph{constructible} representations.
We denote by $\ca_s(P)$ the universal C*-algebra generated by an isometric representation $\{v_p \mid p \in P\}$ of $P$ and projections $\{e_{\Bx} \mid \Bx \in \J\}$ with $e_\mt = 0$ such that:
\[
v_{p_1}^* v_{q_1} \cdots v_{p_n}^* v_{q_n} = e_{\Bx} \qfor \Bx = q_n^{-1} p_n \dots q_1^{-1} p_1 P, p_1^{-1} q_1 \cdots p_n^{-1} q_n = e_G.
\]
It has been shown independently in \cite{KKLL21} and \cite{LS21}, that the canonical $*$-epimorphism from $\ca_s(P)$ to $\T_{\cov}^{\fock}(X)$ is faithful if and only if $\J$ is independent.

It has been noted in \cite{KKLL21} that a non-zero equivariant representation $\{v_p\}_{p \in P}$ of $\ca_s(P)$ descends to a non-zero Fock-covariant representation $\{\dot{v}_p\}_{p \in P}$ of $P$, and thus to an injective representation of $P$.
The crux of the argument in \cite{KKLL21} is that every equivariant isometric representation of $\ca_s(P)$ defines a partial action of $G$ on the spectrum of its fixed point algebra.
Since the fixed point algebra of $\ca_s(P)$ has a smallest one such space, denoted by $\partial \Om_P$, we deduce that such a representation is automatically non-zero.
This is equivalent to having that $1 \notin v(\ker q)$ for the canonical $*$-epimorphism $q \colon \ca_s(P) \to \T_{\cov}^{\fock}(X)$, i.e., $v(X_e) \cap v(\ker q) = (0)$.
\end{example}

We abstract this property of $\ca_s(P)$ in the following definition.
We will later see that it is a necessary and sufficient condition for the form of our results.

\begin{definition}\label{D:separating}
Let $P$ be a unital subsemigroup of a discrete group $G$ and let $X$ be a product system over $P$.
We say that an equivariant representation $t$ of $X$ is \emph{covariant} if it satisfies
\[
t(A) \cap [t_\ast(\ker q_{\cov}^{\fock})]_e = (0),
\]
for the canonical $*$-epimorphism $q_{\cov}^{\fock} \colon \T(X) \longrightarrow \T_{\cov}^{\fock}(X)$.
\end{definition}

Fock-covariant representations are automatically covariant.
Moreover, if $t$ is an equivariant representation of $X$, then Remark \ref{R:induced} implies that there exists an equivariant Fock-covariant representation $\dot{t}$ and a canonical $*$-epimorphism $q_t \colon \ca(t) \to \ca(\dot{t})$ that makes the following diagram
\[
\xymatrix{
\T(X) \ar[rr]^{t_\ast} \ar[d]^{q_{\cov}^{\fock}} & & \ca(t) \ar[d]^{q_t} \\
\T_{\cov}^{\fock}(X) \ar[rr]^{\dot{t}_\ast} & & \ca(\dot{t})
}
\]
commutative, since $\ker q_{\cov}^{\fock}$ is induced, i.e., 
\[
\ker q_t = t_\ast(\ker q_{\cov}^{\fock}) = \sca{t_\ast(\ker q_{\cov}^{\fock}) \bigcap [\ca(t)]_e} = \sca{\ker q_t \bigcap [\ca(t)]_e}.
\]
We now show that covariance allows for injectivity to be preserved both-ways when inducing to Fock-covariant representations.
Therefore, if an injective Fock-covariant representation is terminal for a class of injective representations of $X$ then those need to be covariant.

\begin{proposition}\label{P:separating}
Let $P$ be a unital subsemigroup and let $X$ be a product system over $P$.
Let 
\[
\xymatrix{
\T(X) \ar[rr]^{t_\ast} \ar[d]^{q_{\cov}^{\fock}} & & \ca(t) \ar[d]^{q_t} \\
\T_{\cov}^{\fock}(X) \ar[rr]^{\dot{t}_\ast} & & \ca(\dot{t})
}
\]
be the commutative diagram associated with an equivariant representation $t$ of $X$.
Then $t$ is injective and covariant if and only if $\dot{t}$ is injective.
\end{proposition}

\begin{proof}
Suppose that $t$ is injective and covariant.
If $\dot{t} q_{\cov}^{\fock}(a) = 0$ for $a \in A$, then $q_t t(a) = 0$.
Hence $t(a) \in [\ker q_t]_e = [t_\ast(\ker q_{\cov}^{\fock})]_e$, and covariance yields $t(a) = 0$.
Thus $a = 0$ as $t$ is injective.
For the converse recall that $\la|_A$, and thus $q_{\cov}^{\fock}|_A$, is injective.
Therefore, if $\dot{t}$ is injective, then so is $\dot{t} q_{\cov}^{\fock}|_A$ and hence $t$ is injective.
Moreover, if $t(a) = t_\ast(y)$ for $a \in A \subseteq \T(X)$ and $y \in \ker q_{\cov}^{\fock}$, then
\[
\dot{t} q_{\cov}^{\fock}(a) = q_t t(a) = q_t t_\ast(y) = \dot{t}_\ast q_{\cov}^{\fock}(y) = 0,
\]
and injectivity of $\dot{t} q_{\cov}^{\fock}|_A$ gives $a = 0$, as required.
\end{proof}

\subsection{Fixing the multiplication rules}

Sims--Yeend \cite{SY10} have recognized the following problem with $\la$.
For $b_{\Bx} \in \bo{K}_{\Bx, \la_\ast}$ and $\xi_r \in X_r$, we have $b_{\Bx} \xi_r \in X_r$, and we can define the element $\la(b_{\Bx} \xi_r) \in \T_\la(X)$.
However, $\la( b_{\Bx} \xi_r )$ may not coincide with $b_{\Bx} \la(\xi_r)$.
For example, for $r \notin \Bx$ we have that $\la(b_{\Bx} \xi_r) = 0$, but it may not be that $b_{\Bx} \la(\xi_r) \neq 0$, as it may happen that $rs \in \Bx$ for some other $s \in P$ (consider for example the case of $\ca_\la(P)$).

Sims--Yeend \cite{SY10} suggested a \emph{local} fix towards this end by introducing ideals $I(r, F)$ in $A$, for $r \in P$ and finite $F \subseteq G$, so that $\la(b_{\Bx} \xi_r a_r) = 0$ for all $a_r \in I(r,F)$ and $F \supseteq F(\Bx)$.
Their proposed method covers some but not all cases.
Sehnem \cite{Seh18} ingeniously resolves this for any Fowler product system, and her construction passes in our context as well.

Let $P$ be a unital subsemigroup of a discrete group $G$ and let $X$ be a product system over $P$.
For a finite set $F \subseteq G$ let
\[
K_F := \bigcap_{g \in F} gP.
\]
There is a connection between $\Bx \in \J$ and $F_{\Bx} \subseteq G$ such that $K_{F_{\Bx}} = \Bx$, given in the following way: if $\Bx = q_n^{-1} p_n \dots q_1^{-1} p_1 P$, then $\Bx = K_{F_{\Bx}}$ for
\[
F_{\Bx} := \{e_G, q_n^{-1}, q_n^{-1} p_n, q_n^{-1} p_n q_{n-1}, \dots, q_n^{-1} p_n \dots q_1^{-1} p_1\}.
\]
For $r \in P$ and $g \in F$ define the ideal of $A$ given by
\[
I_{r^{-1} K_{\{r,g\}}} :=
\begin{cases}
\bigcap\limits_{s \in K_{\{r,g\}}} \ker \vphi_{r^{-1} s} & \text{if } K_{\{r,g\}} \neq \mt \text{ and } r \notin K_{\{r,g\}},\\
A & \text{otherwise},
\end{cases}
\]
and set
\[
I_{r^{-1} (r \vee F)} := \bigcap_{g \in F} I_{r^{-1} K_{\{r,g\}}}.
\]
We have that $I_{r^{-1}(r \vee F)} = I_{(pr)^{-1}(pr \vee pF)}$ for all $r,p \in P$, and $I_{r^{-1}(r \vee F)} = I_{(s^{-1}r)^{-1}(s^{-1}r \vee s^{-1}F)}$ for all $r \in sP$.
Moreover
\[
I_{r^{-1}(r \vee F_1)} \subseteq I_{r^{-1}(r \vee F_2)} \; \text{ when } \; F_1 \supseteq F_2.
\]
We declare that $K_\mt= \mt$ and that $I_{r^{-1} (r \vee \mt)} = A$.

\begin{proposition}\label{P:bx acts}
Let $P$ be a unital subsemigroup of a discrete group $G$ and let $X$ be a product system over $P$.

\noindent
\textup{(a)} For $\Bx \in \J$ and $r \in P$ we have the following cases:
\begin{enumerate}
\item If $r \in \Bx$, then
\[
b_{\Bx} \la(\xi_r) = \la( b_{\Bx} \xi_r ) \foral b_{\Bx} \in \bo{K}_{\Bx, \la_\ast}, \xi_r \in X_r.
\]
\item If $r \notin \Bx$ and $F \supseteq F_{\Bx}$, then 
\[
b_{\Bx} \la(\xi_r a_r) = \la( b_{\Bx} \xi_r a_r) = 0 \foral b_{\Bx} \in \bo{K}_{\Bx, \la_\ast}, \xi_r \in X_r, a_r \in I_{r^{-1}(r \vee F)}.
\]
\end{enumerate}

\noindent
\textup{(b)} If $t$ is a Fock-covariant representation, then for $b_{\Bx} \in \bo{K}_{\Bx, \la_\ast}$ and $F \supseteq F_{\Bx}$ we have that
\[
t_{\ast}(\iota(b_{\Bx})) t(\xi_r a_r) = t(b_{\Bx} \xi_r a_r) \foral \xi_r a_r \in X_F,
\]
for the induced $*$-representation $t_{\ast} \colon \T_{\cov}^{\fock}(X) \to \ca(t)$.
\end{proposition}

\begin{proof}
(a) For item (i), if $r \in \Bx$, then we have that
\[
\la(\xi_{p_1})^* \la(\xi_{q_1}) \cdots \la(\xi_{p_n})^* \la(\xi_{q_n}) \xi_r
=
\xi_{p_1}^* \xi_{q_1} \cdots \xi_{p_n}^* \xi_{q_n} \xi_r \in X_{p_1}^* \cdot X_{q_1} \cdots X_{p_n}^* \cdot X_{q_n} \cdot X_r \subseteq X_r,
\]
for $p_i, q_i \in P$ such that $\Bx = q_n^{-1} p_n \dots q_1^{-1} p_1 P$ and $p_1^{-1} q_1 \cdots p_n^{-1} q_n = e_G$.
Moreover, since $r \in \Bx$, then $rq \in \Bx$ for all $q \in P$, and a similar computation gives that
\begin{align*}
\la(\xi_{p_1}^* \xi_{q_1} \cdots \xi_{p_n}^* \xi_{q_n} \xi_r) \xi_q
& =
\xi_{p_1}^* \xi_{q_1} \cdots \xi_{p_n}^* \xi_{q_n} \xi_r \xi_q
=
\la(\xi_{p_1})^* \la(\xi_{q_1}) \cdots \la(\xi_{p_n})^* \la(\xi_{q_n}) \la(\xi_r) \xi_q.
\end{align*}
By applying to linear combinations and then passing to limits we obtain the required equality for all $b_{\Bx} \in \bo{K}_{\Bx, \la_\ast}$.

For item (ii), if $\Bx = \mt$ then by Proposition \ref{P:Fock cov rel} we have that $b_{\Bx} = 0$, and thus the equality is trivially satisfied.
So suppose that $\Bx \neq \mt$.
We have $b_{\Bx} \xi_p = 0$ for all $p \notin \Bx$, by applying first on elementary generators of $\bo{K}_{\Bx, \la_\ast}$, and then extending to limits of finite sums.
Thus, if $r \notin \Bx$, then we get that 
\[
\la(b_{\Bx} \xi_r a_r) = \la(0 \cdot a_r) = 0 \foral \xi_r \in X_r, a_r \in I_{r^{-1}(r \vee F)}.
\]
Let now $F \supseteq F_{\Bx}$; we will show that 
\[
b_{\Bx} \xi_r a_r \xi_q = 0 \foral q \in P, \xi_r \in X_r, a_r \in I_{r^{-1}(r \vee F)}.
\]
We consider the following cases.

\smallskip

\noindent
Case 1. If $rq \notin \Bx$, we then have that $b_{\Bx} \xi_r a_r \xi_q = 0$, as $\xi_r a_r \xi_q \in X_{rq}$.

\smallskip

\noindent
Case 2. If $rq \in \Bx$, then $rq \in gP$ for all $g \in F_{\Bx}$, i.e., $rq \in K_{\{r,g\}}$ for all $g \in F_{\Bx}$.
Since $r \notin \Bx$, then there exists a $g' \in F_{\Bx}$ such that $r \notin K_{\{r,g'\}}$ but for which $rq \in K_{\{r, g'\}}$.
Hence by definition we have that
\begin{align*}
a_r \in I_{r^{-1}(r \vee F)} 
& \subseteq I_{r^{-1}(r \vee F_{\Bx})} 
= \bigcap_{g \in F_{\Bx}} \bigcap_{s \in K_{\{r, g\}}} \ker \vphi_{r^{-1} s} 
\subseteq \ker \vphi_{r^{-1} (r q)} = \ker \vphi_{q}.
\end{align*}
Therefore $a_r \xi_q = 0$ for all $\xi_q \in X_q$, from which we deduce that $b_{\Bx} \xi_r a_r \xi_q = 0$.

\smallskip

\noindent
(b) Let $b_{\Bx} \in \bo{K}_{\Bx, \la_\ast}$ and $\xi_r a_r \in X_F$ with $F \supseteq F_{\Bx}$.
From item (a,ii) we have that
\begin{align*}
t(b_{\Bx} \xi_r a_r)
& =
t_\ast \iota \la(b_{\Bx} \xi_r a_r)
=
t_\ast(\iota(b_{\Bx}) \iota \la(\xi_r a_r) )
=
t_\ast(\iota(b_{\Bx})) t_\ast \iota \la(\xi_r a_r)
=
t_\ast(\iota(b_{\Bx})) t(\xi_r a_r),
\end{align*}
as required.
\end{proof}

As an immediate consequence we have that the $\bo{K}$-cores embed injectively in every equivariant injective representation of $\F\C_G X$.
This is the analogue of the embedding of compacts in the right LCM-semigroup setting, and explains the use of $\bo{K}$- for denoting these cores.

\begin{proposition}\label{P:K-injective}
Let $P$ be a unital subsemigroup of a discrete group $G$ and let $X$ be a product system over $P$.
If $t$ is an equivariant injective Fock-covariant representation of $X$, then $t_\ast$ is injective on the $\bo{K}$-cores of $[\T_{\cov}^{\fock}(X)]_e$.
\end{proposition}

\begin{proof}
Let $\mt \neq \Bx \in \J$.
To reach a contradiction let $0 \neq b_{\Bx} \in \bo{K}_{\Bx, \la_\ast}$ such that $t_\ast( \iota(b_{\Bx}) ) = 0$.
Then there exists a $\xi_r \in X_r$ such that $b_{\Bx} \xi_r \neq 0$.
Hence $r \in \Bx$, and by Proposition \ref{P:bx acts} item (a,i) we have that
\[
0 = t_\ast( \iota( b_{\Bx}) ) t(\xi_r) = t_\ast( \iota(b_{\Bx}) \iota\la(\xi_r) ) = t_\ast \iota \la(b_{\Bx} \xi_r) = t(b_{\Bx} \xi_r).
\]
Since $t$ is injective on $X$, we thus have that $b_{\Bx} \xi_r = 0$ which is a contradiction.
\end{proof}

\subsection{Strongly covariant representations}

We can now use the ideals $I_{r^{-1}(r \vee F)}$ to build an augmented Fock space and then pass to the strongly covariant setting in the same spirit with \cite{Seh18, SY10}.
The proofs of \cite[Sections 3.1 and 3.2]{Seh18} require only the algebraic properties of the Fock representation which we have established here. 
For our purposes we will follow the approach of \cite{DKKLL20} in this endeavour.

Recall the definition of $I_{r^{-1}(r \vee F)}$ for a finite $F \subseteq G$ and $r \in P$ from the previous subsection, and let the C*-correspondences
\[
X_F := \sumoplus_{r \in P} X_r I_{r^{-1} (r \vee F)}
\qand
X_F^+ := \sumoplus_{g \in G} X_{gF}.
\]
We declare that $X_\mt = X_\mt^+ = \F X$.
Each $X_F^+$ is reducing for the coaction
\[
\de \colon \T_\la(X) \longrightarrow \T_\la(X) \otimes \ca_{\max}(G); \la(\xi_p) \mapsto \la(\xi_p) \otimes u_p,
\]
giving rise to a $*$-representation
\[
\Phi_F \colon \T_\la(X) \longrightarrow \L( X_F^+ ) ; \la(\xi_p) \mapsto (\la(\xi_p) \otimes u_p)|_{X_F^+}.
\]
Here we make the identification
\[
X_{gF} \longrightarrow X_{gF} \otimes \de_g ; \xi_r a_r \mapsto \xi_r a_r \otimes \de_g \qfor \xi_r \in X_r, a_r \in I_{r^{-1}(r \vee gF)}.
\]
Moreover $X_F \subseteq X_F^+$ is reducing for $[\T(X)]_e$ and so we obtain the representation
\[
\bigoplus\limits_{\textup{fin } F \subseteq G} \Phi_F(\cdot)|_{X_F} \colon [\T_\la(X)]_e \longrightarrow \prod\limits_{\text{fin } F \subseteq G} \L(X_F).
\]
We fix the ideal $\I_{\scv, e}$ in $[\T(X)]_e$ by using the corona universe, namely
\[
b \in \I_{\scv, e} \qiff \bigoplus\limits_{\textup{fin } F \subseteq G} \Phi_F(\la_\ast(b))|_{X_F} \in c_0(\L(X_F) \mid \textup{fin } F \subseteq G).
\]

\begin{definition}
Let $P$ be a unital subsemigroup of a discrete group $G$ and let $X$ be a product system over $P$.
Suppose that $\T(X) = \ca(\wt{t})$ for a representation $\wt{t}$ of $X$ and let the induced ideal 
\[
\I_{\scv} := \sca{ \I_{\scv, e} }.
\]
We write $A \times_X P$ for the equivariant quotient of $\T(X)$ by $\I_{\scv}$.
We define the \emph{strongly covariant} bundle of $X$ be the Fell bundle
\[
\S\C_G X := \Big\{ [A \times_X P]_g \Big\}_{g \in G},
\]
given by the coaction of $G$ on $A \times_X P$.
A representation of $\S\C_G X$ will be called a \emph{strongly covariant representation of $X$}.
\end{definition}

\begin{remark}
With this notation it is immediate that $A \times_X P$ is $\ca_{\max}(\S\C_G X)$.
We will also use the notation $A \times_{X, \la} P$ for the reduced $\ca_\la(\S\C_G X)$.
Eventually we will show that those are independent from the group embedding.
\end{remark}

Using this presentation of strong covariance from \cite{DKKLL20} we can immediately show that strongly covariant representations are Fock-covariant.
This was shown only for specific examples in \cite{Seh18}.

\begin{proposition} \label{P:sc is fc}
Let $P$ be a unital subsemigroup of a discrete group $G$ and let $X$ be a product system over $P$.
Then the canonical $*$-epimorphism $\T(X) \to A \times_X P$ factors through $\T(X) \to \T_{\cov}^{\fock}(X)$.
\end{proposition}

\begin{proof}
Let $q \colon \T(X) \to A \times_X P$ be the canonical quotient map by the ideal $\sca{\I_{\scv, e}}$.
Also recall the canonical $*$-epimorphisms
\[
\T(X) \longrightarrow \T_{\cov}^{\fock}(X) \longrightarrow \T_\la(X).
\]
By definition, if $x \in \ker\la_\ast \bigcap [\T(X)]_e$ then $x \in \ker q \bigcap [\T(X)]_e$.
Moreover, if $x \in \ker \la_\ast \bigcap [\T(X)]_g$ then $x^*x \in \ker \la_\ast \bigcap [\T(X)]_e$, and thus $x \in \ker q \bigcap [\T(X)]_g$.
This induces a canonical Fell bundle homomorphism that completes the diagram
\[
\xymatrix{
\{ [\T(X)]_g \}_{g \in G} \ar[rr]^q \ar[dr]_{\la} & & \S\C_G X \\
& \F \C_G X \ar@{.>}[ur] &
}
\]
and thus produces a canonical $*$-epimorphism $\T_{\cov}^{\fock}(X) \longrightarrow \ca_{\max}(\S\C_G X) = A \times_X P$.
\end{proof}

In \cite[Proposition 3.5]{Seh18} it is proven that $A \hookrightarrow A \times_X P$ faithfully.
This elegant proof applies directly here.
Indeed the only place that requires some attention is the argument in \cite[Proof of Proposition 3.5]{Seh18} that if 
\[
\sca{\vphi_{r_1}(a) \xi_{r_1}, \vphi_{r_1}(a) \xi_{r_1}} \notin \ker \vphi_{r_1^{-1} r_2},
\]
then $a \notin \ker \vphi_{r_2}$.
This still holds in our setting since $X_{r_1} \cdot X_{r_1^{-1} r_2} \subseteq X_{r_1 r_1^{-1} r_2} = X_{r_2}$.
Let us include the details for the sake of completeness.

\begin{proposition} \label{P:A emb}
Let $P$ be a unital subsemigroup of a discrete group $G$ and let $X$ be a product system over $P$.
Then 
\[
\la(A) \bigcap c_0(\L(X_F) \mid \textup{finite } F \subseteq G) = (0).
\]
\end{proposition}


\begin{proof}
It suffices to show that $\Phi_F(\la(a))|_{X_F} \neq 0$ for every finite $F \subseteq G$, so that $\nor{\Phi_F(a)|_{X_F}} = \nor{a}$.
We move in steps.
If $a I_{e \vee F} \neq (0)$, then there is nothing to show.
If $a I_{e \vee F} = (0)$, then there exists a $g_1 \in F$ such that $a \notin I_{K_{\{e, g_1\}}}$.
Thus there exists an $r_1 \in P \cap g_1 P$ with $\vphi_{r_1}(a) \neq 0$.
Set
\[
F_1 := \{g \in F \mid K_{\{r_1, g\}} \neq \mt \textup{ and } r_1 \notin K_{\{r_1, g\}} \}.
\]
Then $g_1 \notin F_1$ and $F_1 \subsetneq F$, satisfying
\[
X_{r_1} I_{r_1^{-1}(r_1 \vee F)} = X_{r_1} I_{r_1^{-1}(r_1 \vee F_1)}.
\]
If $F_1 = \mt$, then $I_{r_1^{-1}(r_1 \vee F_1)} = A$, and so $X_{r_1} I_{r_1^{-1}(r_1 \vee F)} = X_{r_1}$.
Since $\vphi_{r_1}(a) \neq 0$, we are done.
Likewise if $\vphi_{r_1}(a)|_{X_{r_1} I_{r_1^{-1}(r_1 \vee F_1)}} \neq 0$.

Suppose that $F_1 \neq \mt$ and that $\vphi_{r_1}(a)|_{X_{r_1} I_{r_1^{-1}(r_1 \vee F_1)}} = 0$.
Then there exists a $g_2 \in F_1$ and a $\xi_{r_1} \in X_{r_1}$ so that
\[
\sca{\vphi_{r_1}(a) \xi_{r_1}, \vphi_{r_1}(a) \xi_{r_1}} \notin I_{r_1^{-1}K_{\{r_1, g_2\}}}.
\]
Therefore there exists an $r_2 \in K_{\{r_1, g_2\}}$ such that
\[
\sca{\vphi_{r_1}(a) \xi_{r_1}, \vphi_{r_1}(a) \xi_{r_1}} \notin \ker \vphi_{r_1^{-1} r_2}.
\]
From this we derive that $a \notin \ker \vphi_{r_2}$, since 
\[
X_{r_1} \cdot X_{r_1^{-1} r_2} \subseteq X_{r_1 r_1^{-1} r_2} = X_{r_2}.
\]
This is the only place where the proof differs from that of \cite[Proposition 3.5]{Seh18}.
Let now
\[
F_2 := \{g \in F \mid K_{\{r_2, g\}} \neq \mt \textup{ and } r_2 \notin K_{\{r_2, g\}} \},
\]
for which we have that $F_2 \subsetneq F_1$ and that
\[
X_{r_2} I_{r_2^{-1}(r_2 \vee F)} = X_{r_2} I_{r_2^{-1}(r_2 \vee F_2)}.
\]
If $F_2 = \mt$ or $\vphi_{r_2}(a)|_{X_{r_2} I_{r_2^{-1}(r_2 \vee F_2)}} \neq 0$, then we are done.
Otherwise we can deduce an $F_3 \subsetneq F_2$ to reduct.
Since $F$ is finite this process stops at some $r_k$ with $\vphi_{r_k}(a)|_{X_{r_k} I_{r_k^{-1}(r_k \vee F)}} \neq 0$, and the proof is complete.
\end{proof}

The Bessel inequality of Corollary \ref{C:Bessel} and Fock-covariance of Proposition \ref{P:sc is fc} provide a property of $\ker q_{\scv}$ that is key for both \cite[Lemma 3.6]{Seh18} and \cite[Theorem 4.1]{CLSV11}.

\begin{proposition} \label{P:fpa t}
Let $P$ be a unital subsemigroup of a discrete group $G$ and let $X$ be a product system over $P$.
If $t$ is an injective Fock-covariant representation of $X$, then
\[
\ker t_\ast \bigcap [\T_{\cov}^{\fock}(X)]_e \subseteq \ker q_{\scv} \bigcap [\T_{\cov}^{\fock}(X)]_e
\]
for the canonical $*$-epimorphism $q_{\scv} \colon \T_{\cov}^{\fock}(X) \to A \times_X P$.
\end{proposition}

\begin{proof}
Let $b \in \ker t_\ast \bigcap [\T_{\cov}^{\fock}(X)]_e$.
Our aim is to show that for $\eps >0$ there exists a finite $F_0 \subseteq G$ such that $\| \Phi_F \la_\ast(b) |_{X_F} \| < \eps$ for all finite $F \supseteq F_0$.

Recall that $\bo{K}_{\mt, \la_\ast} = (0)$ by Proposition \ref{P:Fock cov rel}.
Thus by Proposition \ref{P:ind lim} for $\eps>0$ there exist non-empty $\Bx_1, \dots, \Bx_n \in \J$ such that
\[
b' := \sum_{j = 1}^n \iota(b_{\Bx_j}) 
\; \text{ with } \; 
b_{\Bx_j} \in \bo{K}_{\Bx_j, \la_\ast} 
\; \text{ such that } \;
\nor{b - b'} < \frac{\eps}{2}.
\]
In particular, we have that
\[
\| t_\ast(b') \| = \| t_\ast(b - b') \| \leq \| b - b' \| \leq \frac{\eps}{2}.
\]
Let $F \supseteq F_0 := \bigcup_{j=1}^n F_{\Bx_j}$.
For a fixed $\xi = \sum_{r \in P}^{\oplus} \xi_r a_r \in X_F$ set $\eta := \sum_{r \in P} t(\xi_r a_r)$.
Since $t$ is injective on $A$ we have that
\[
\nor{\eta}_{Y_{t}}^2 = \| \sum_{r \in P} t(\xi_r a_r)^* t(\xi_r a_r) \| = \| \sum_{r \in P} \sca{\xi_r a_r, \xi_r a_r}_A \| = \nor{\xi}_{X_F}^2.
\]
Since $F \supseteq F_{\Bx_j}$ for all $j=1, \dots, n$, by using Proposition \ref{P:bx acts} item (b) we see that 
\begin{align*}
t_\ast( \iota(\la_\ast(b') \xi_r a_r) )
& =
t_\ast(b') t(\xi_r a_r) 
=
\sum_{j=1}^n t_\ast( \iota(b_{\Bx_j}) ) t(\xi_r a_r) 
=
\sum_{j=1}^n t(b_{\Bx_j} \xi_r a_r)
\in 
\sum_p t(X_p).
\end{align*}
Consequently $t_\ast(b') \eta$ is in $\sum_p t(X_p)$.
Thus we can apply injectivity of $t$ on $A$ and Bessel's inequality from Corollary \ref{C:Bessel} to obtain
\begin{align*}
\|\Phi_F(\la_\ast(b')) \xi\|_{X_F}
& =
\| \la_\ast(b') \xi \|_{\F X}
=
\| | \sum_{j, r} \iota(b_{\Bx_j}) \xi_r a_r |^{1/2} \|_{\F X} \\
& =
\| | \sum_{j, r} t_\ast \iota(b_{\Bx_j}) t_\ast(\xi_r a_r) |^{1/2} \|_{Y_t} 
=
\| t_\ast(b') \eta \|_{Y_{t}}
\leq
\nor{t_\ast(b')} \cdot \nor{\eta}_{Y_{t}}
\leq 
\frac{\eps}{2} \cdot \nor{\xi}_{X_F}.
\end{align*}
Hence we derive that $\|\Phi_F \la_\ast(b')|_{X_F}\| \leq \eps/2$ and thus
\[
\| \Phi_F \la_\ast (b)|_{X_F} \| \leq \| \Phi_F \la_\ast(b - b')|_{X_F} \| + \|\Phi_F \la_\ast(b')|_{X_F}\|
\leq \nor{b - b'} + \frac{\eps}{2} < \eps, 
\]
for all $F \supseteq F_0$, as required.
\end{proof}

We can now extend \cite[Lemma 3.6]{Seh18} in our setting.

\begin{proposition} \label{P:lift A to fpa}
Let $P$ be a unital subsemigroup of a discrete group $G$ and let $X$ be a product system over $P$.
A $*$-representation of $A \times_X P$ is injective on $[A \times_X P]_e$ if and only if it is injective on $A \hookrightarrow A \times_X P$.
\end{proposition}

\begin{proof}
One direction is immediate since $A \hookrightarrow [A \times_X P]_e$ by Proposition \ref{P:A emb}.
Conversely suppose that $t_\ast$ is a $*$-representation of $\T(X) = \ca(\wt{t})$ that factors through $A \times_X P$.
Without loss of generality, and due to Proposition \ref{P:sc is fc}, we can assume that $t$ is a representation of $\T_{\cov}^{\fock}(X)$ and thus Fock-covariant.
Proposition \ref{P:fpa t} shows that, if $b \in [\T_{\cov}^{\fock}(X)]_e$ is such that $t_\ast(b) = 0$, then $q_{\scv}(b) = 0$, as required.
\end{proof}

Finally, one of the important points of Sehnem's theory is that $A \times_X P$ does not depend on the group $G$ where $P$ embeds \cite[Lemma 3.9]{Seh18}.
The proof in our case is the same, given that we have the analogues of \cite[Proposition 3.5 and Lemma 3.6]{Seh18}.
We include it here as we will need to add a comment and also be explicit that this independence descends also to the reduced version $A \times_{X, \la} P$ which is not discussed in \cite{Seh18}.

\begin{proposition}\label{P:K emb}
Let $P$ be a unital subsemigroup and let $X$ be a product system over $P$.
Let $G$ be the enveloping group of $P$ and suppose that $\ga \colon G \to P$ is a group homomorphism that is injective on $P$.
If $r \in P$ and $g \in G$ satisfy $K_{\{r, g\}} \neq \mt$ and $r \notin K_{\{r, g\}}$, then $K_{\{\ga(r), \ga(g)\}} \neq \mt$ and $\ga(r) \notin K_{\{\ga(r), \ga(g)\}}$.
\end{proposition}

\begin{proof}
It is clear that with these assumptions we have that $K_{\{\ga(r), \ga(g)\}} \neq \mt$ and also there are $q_1, q_2 \in P$ such that $r q_1 = g q_2$ with $q_2 q_1^{-1} \notin P$.
If it were that $\ga(r) \in K_{\{\ga(r), \ga(g)\}}$ then $\ga(r) = \ga(r q_1 q_2^{-1} s)$ for some $s \in P$, and thus that $\ga(q_2) = \ga(s q_1)$.
As $\ga|_P$ is one-to-one we get the contradiction $q_2 q_1^{-1} = s \in P$.
\end{proof}

Although the details of this step are omitted from the proof of \cite[Lemma 3.9]{Seh18}, we find that they elucidate the independence from the group.
Practically the elements of the group to be considered come from $P q_1 q_2^{-1}$ with $q_1 q_2^{-1} \notin P^{-1}$, and this subset of $P \cdot P^{-1}$ cannot give non-trivial relations in a group generated by $P$.

\begin{proposition}\label{P:ind sc}
Let $P$ be a unital subsemigroup and let $X$ be a product system over $P$.
Suppose that $P$ admits two group embeddings $i_G \colon P \to G$ and $i_H \colon P \to H$.
Then there exists a $*$-isomorphism
\[
\ca_{\max}(\S\C_G X) \longrightarrow \ca_{\max}(\S\C_H X); \xi_{i_G(p)} \mapsto \xi_{i_H(p)},
\]
which descends to $*$-isomorphism
\[
\ca_{\la}(\S\C_G X) \longrightarrow \ca_{\la}(\S\C_H X); \xi_{i_G(p)} \mapsto \xi_{i_H(p)}.
\]
In particular, there exists a Fell bundle isomorphism $\S\C_G X \to \S\C_H X$ such that
\[
[A \times_X P]_{i_G(p_1)^{-1} i_G(q_1) \cdots i_{G}(p_n)^{-1} i_G(q_n)} \longrightarrow [A \times_X P]_{i_H(p_1)^{-1} i_H(q_1) \cdots i_{H}(p_n)^{-1} i_H(q_n)},
\]
isometrically for every $p_1, q_1, \dots, p_n, q_n \in P$.
\end{proposition}

\begin{proof}
Let $G$ be the enveloping group of $P$ and $\ga \colon G \to H$ be the canonical epimorphism that fixes $P$.
Let us denote by $\I_{\scv, e_G}$ and by $\I_{\scv, e_H}$ the strong covariance relations that define the $G$-strongly covariant and the $H$-strongly covariant ideals.
We will show that $\I_{\scv, e_G} \subseteq \I_{\scv, e_H}$.
Recall that these ideals are kernels of the maps
\[
[\T(X)]_{e_G} \longrightarrow [\T_\la(X)]_{e_G} \longrightarrow \quo{\prod\limits_{\textup{fin } F \subseteq G} \L(X_F)}{c_0(\L(X_F) \mid \textup{fin } F \subseteq G)},
\]
and
\[
[\T(X)]_{e_H} \longrightarrow [\T_\la(X)]_{e_H} \longrightarrow \quo{\prod\limits_{\textup{fin } \ga(F) \subseteq H} \L(X_{\ga(F)})}{c_0(\L(X_{\ga(F)}) \mid \textup{fin } \ga(F) \subseteq H)}.
\]
By Proposition \ref{P:la fpa same} we have that $[\T_\la(X)]_{e_G} = [\T_\la(X)]_{e_H}$, and thus it suffices to show that $X_{\ga(F)} \subseteq X_F$ for every finite $F \subseteq G$.
Towards this end it suffices to show that
\[
I_{\ga(r)^{-1} K_{\{\ga(r), \ga(g)\}}} \subseteq I_{r^{-1} K_{\{r,g\}}} \foral r \in P, g \in G.
\]
If $K_{\{r,g\}} = \mt$, or if $r \in K_{\{r,g\}}$, then $I_{r^{-1} K_{\{r,g\}}} = A \supseteq I_{\ga(r)^{-1} K_{\{\ga(r), \ga(g)\}}}$.
If $K_{\{r,g\}} \neq \mt$ and $r \notin K_{\{r,g\}}$ then Proposition \ref{P:K emb} yields $K_{\{\ga(r), \ga(g)\}} \neq \mt$ and $\ga(r) \notin K_{\{\ga(r), \ga(g)\}}$.
Since $K_{\{r,g\}} \subseteq P$ and $K_{\{\ga(r), \ga(g)\}} \subseteq \ga(P)$ then $\ga$ defines a bijection between these two sets.
Thus we have
\begin{align*}
I_{\ga(r)^{-1} K_{\{\ga(r), \ga(g)\}}}
& =
\bigcap_{\ga(s) \in K_{\{\ga(r), \ga(g)\}}} \ker \phi_{\ga(r)^{-1} \ga(s)} 
=
\bigcap_{s \in K_{\{r, g\}}} \ker \phi_{r^{-1} s} 
= 
I_{r^{-1} K_{\{r,g\}}}.
\end{align*}

Therefore, there exists a $*$-epimorphism
\[
\Phi_\ga \colon \ca_{\max}(\S\C_G X) \longrightarrow \ca_{\max}(\S\C_H X)
\]
that fixes every $X_p$, and thus $\Phi_\ga ([\S\C_G X]_g) \subseteq [\S\C_H X]_{\ga(g)}$.
Hence it descends to a $*$-epimorphism on the reduced C*-algebras
\[
\dot{\Phi}_\ga \colon \ca_{\la}(\S\C_G X) \longrightarrow \ca_{\la}(\S\C_H G).
\]
By Proposition \ref{P:A emb} $\dot{\Phi}_\ga|_A$ is faithful, and Proposition \ref{P:lift A to fpa} yields that $\dot{\Phi}_\ga$ is faithful on $[\S \C_G X]_e$.
By construction $\dot{\Phi}_\ga$ intertwines the faithful conditional expectations, and so $\dot{\Phi}_\ga$ is faithful.
\end{proof}

\section{Co-universality}

\subsection{Boundary represenations}

Our goal is to show that the reduced strongly covariant representation is terminal for the class of equivariant injective covariant represenations (and thus also for the corresponding Fock-covariant).
The notion of terminal representations appears in several contexts and constructions of C*-algebras, sometimes under disguise.
Let us set a general framework and provide some examples.

\begin{definition}
Let $\F$ be a class of $*$-representations of a C*-algebra.
We say that $\Phi$ is \emph{$\F$-boundary} if and only if $\Phi$ is a member of the class $\F$, and $\Phi$ factors through every member $\Psi$ of the class $\F$.
\end{definition}

Boundary representations generate unique \emph{co-universal C*-algebras}.
We will limit the amount of detail in the following examples.
The interested reader may refer to the cited works for more information.

\begin{example}
Let $A, B$ be C*-algebras.
If $\nor{\cdot}_{\ga}$ is a C*-norm on the algebraic tensor product $A \odot B$, then we write $A \otimes_\ga B$ for the $\nor{\cdot}_{\ga}$-completion of $A \odot B$.
By considering ``all'' the injective $*$-algebraic representations of $A \odot B$ on Hilbert spaces, we obtain the maximal C*-norm on $A \odot B$ that dominates all other C*-norms.
Then the spatial tensor product norm is boundary for the $*$-representations of $A \otimes_{\max} B$ that are injective on $A \odot B$.
\end{example}

\begin{example}
Let $I \lhd A$ be a two-sided closed ideal in a C*-algebra $A$ and consider the short exact sequence
\[
\ses{I}{}{A}{}{A/I}.
\]
A short exact sequence $*$-homomorphism of $I \lhd A$ corresponds to a pair $I' \lhd A'$ that makes the following diagram
\[
\xymatrix{
0 \ar[r] & I \ar[r] \ar[d] & A \ar[r] \ar[d] & A/I \ar[r] \ar[d] & 0 \\
0 \ar[r] & I' \ar[r] & A' \ar[r] & A'/I' \ar[r] & 0
}
\]
commutative.
The maximal tensor product is exact in the sense that it gives a short exact sequence
\[
\ses{I \otimes_{\max} B}{}{A \otimes_{\max} B}{}{A/I \otimes_{\max} B}
\]
for all C*-algebras $B$.
If $A/I$ is nuclear or if $B$ is exact, then the pair $I \otimes B \lhd A \otimes B$ induces a boundary representation for the short exact sequence $*$-homomorphisms of $I \otimes_{\max} B \lhd A \otimes_{\max} B$ that are injective on $A \odot B$.
\end{example}

\begin{example}
Let $\fA$ be an operator algebra.
Let $\F_{\fA}$ be the class of the $*$-representations of $\ca_{\max}(\fA)$ that are completely isometric on $\fA$.
Then the canonical $*$-representation $\ca_{\max}(\fA) \to \cenv(\fA)$ that fixes $\fA$ is an $\F_{\fA}$-boundary representation.
This follows from the existence of the C*-envelope by Hamana \cite{Ham79}.
A similar result holds for operator systems, i.e., unital selfadjoint subspaces of some $\B(H)$.
The existence of the maximal C*-algebra of an operator system is shown by Kirchberg--Wasserman \cite{KW98}, and the existence of the C*-envelope of an operator system is shown by Arveson \cite{Arv08}, simplifying the arguments of Dritschel-McCullough \cite{DM05}.
For non-unital operator systems the reader may refer to \cite{KKM21}.
\end{example}

\begin{example}
Let $\fA$ be an operator algebra and $\al$ be an action by a locally compact Hausdorff group $G$ on $\fA$ by completely isometric automorphisms.
Then $\al$ extends to $\ca_{\max}(\fA)$ and to $\cenv(\fA)$.
Let $\F_{\fA}^{\al}$ be the class of the $*$-representations of $\ca_{\max}(\fA)$ that are completely isometric on $\fA$ and admit an equivariant action of $G$.
Then the $*$-representation $\ca_{\max}(\fA) \to \cenv(\fA)$ that fixes $\fA$ is an $\F_{\fA}^{\al}$-boundary representation.
This follows by the C*-envelope property and the automatic extension of $\al$ on $\cenv(\fA)$.
\end{example}

\begin{example}
Let $\fA$ be an operator algebra and $\de$ be a coaction of a discrete group $G$ on $\fA$.
Then it induces a coaction on $\ca_{\max}(\fA)$.
Let $\F_{\fA}^{\de}$ be the class of the $*$-representations of $\ca_{\max}(\fA)$ that are completely isometric on $\fA$ and admit an equivariant coaction of $G$.
In \cite{DKKLL20} it is shown that the $*$-representation $\ca_{\max}(\fA) \to \cenv(\fA) \otimes \ca_{\max}(G)$ that fixes $\fA$ is an $\F_{\fA}^{\de}$-boundary representation.
\end{example}

\begin{example}
Let $\B = \{\B_g\}_{g \in G}$ be a Fell bundle over a discrete group $G$.
Let $\F_\B^\de$ be the class of representations of $\ca_{\max}(\B)$ that are injective on $\B_e$ and admit an equivariant coaction of $G$.
Then the left regular representation $\la \colon \ca_{\max}(\B) \to \ca_\la(\B)$ is $\F_\B^\de$-boundary.
This is the co-universal result of Exel \cite{Exe97}.
Fell bundles encode a great number of examples such as C*-crossed products, partial crossed products and Hilbert bimodules.
\end{example}

\begin{example}
Let $P$ be a subsemigroup of a discrete group of $G$.
Let $\F_s^\de$ be the class of constructible non-zero representations of $\ca_s(P)$ in the sense of Li \cite{Li12} that admit an equivariant coaction of $G$.
In \cite{KKLL21} it is shown that the canonical $*$-representation $\ca_s(P) \to \mathrm{C}(\partial \Om_P) \rtimes_{r} G$ that fixes $P$ is an $\F_s^\de$-boundary representation.
Here $\partial \Om_P$ is the smallest $G$-invariant subspace in the spectrum of $[\ca_s(P)]_e$.
Moreover, it is shown that $\mathrm{C}(\partial \Om_P) \rtimes_{r} G \simeq \cenv(\A(P), G, \de)$.
\end{example}

The following four examples concern successively stronger results, but each one with an essentially different approach.

\begin{example}
Let $X$ be a single C*-correspondence.
We write $\T(X)$ for the universal C*-algebra with respect to the representations of $X$ and $\O(X)$ for the universal C*-algebra with respect to the covariant representations of $X$ in the sense of Katsura \cite{Kat04}.
Let $\F_{X}^{\bT}$ be the class of representations of $\T(X)$ that are injective on $X$ and admit an equivariant action of $\bT$.
Then the Cuntz--Pimsner representation $\T(X) \to \O(X)$ is an $\F_X^{\bT}$-boundary representation.
This follows from the study of the gauge-invariant ideals of $\T(X)$ by Katsura \cite{Kat07}.
This class contains (topological) graph C*-algebras and C*-algebras of $\bZ$-dynamics. 
\end{example}

\begin{example}
Let $(G,P)$ be a quasi-lattice.
Let $X$ be a compactly aligned product system over $P$ in the sense of Fowler \cite{Fow02}.
We write $\N\T(X)$ for the universal C*-algebra with respect to the Nica-covariant representations of $X$ and $\N\O(X)$ for the universal C*-algebra with respect to the Cuntz--Nica--Pimsner representations of $X$ in the sense of Sims--Yeend \cite{SY10}.
Then $\N\O(X)$ is equivariant and let $\N\O(X)^r$ be the reduced C*-algebra of the bundle induced by this coaction.
Let $\F_{\N X}^{\de}$ be the class of representations of $\N\T(X)$ that are injective on $X$ and admit an equivariant coaction of $G$.
Carlsen--Larsen--Sims--Vittadello \cite{CLSV11} show that the Cuntz--Nica--Pimsner representation $\N\T(X) \to \N\O(X)^r$ is an $\F_{\N X}^{\de}$-boundary representation, when $X$ is injective, or when $P$ is directed and $X$ is $\wt{\phi}$-injective.
This follows by studying the structure of the augmented Fock space representation of Sims--Yeend \cite{SY10}.
This class contains the case $P = \bZ_+^d$, e.g., C*-algebras from higher rank (topological) graphs.
\end{example}

\begin{example}
Let $(G,P)$ be an abelian lattice.
Let $X$ be a compactly aligned product system over $P$ in the sense of Fowler \cite{Fow02}.
We write $\N\T(X)$ for the universal C*-algebra with respect to the Nica-covariant representations of $X$.
Let $\F_{\N X}^{\wh{G}}$ be the class of representations of $\N\T(X)$ that are injective on $X$ and admit an equivariant action of the dual group $\wh{G}$.
In \cite{DK20} it is shown that the representation $\N\T(X) \to \cenv(\T_\la(X)^+)$ is an $\F_{\N X}^{\wh{G}}$-boundary representation.
Furthermore, it is shown that $\cenv(\T_\la(X)^+)$ coincides with $A \times_X P$.
The latter follows by the co-universal properties of $A \times_X P$ and $\cenv(\T_\la(X)^+)$.
\end{example}

\begin{example}
Let $P$ be a right LCM semigroup that is a subsemigroup of a discrete group $G$.
Let $X$ be a compactly aligned product system over $P$ in the sense of Kwa\'{s}niewski--Larsen \cite{KL19b}.
We write $\N\T(X)$ for the universal C*-algebra with respect to the Nica-covariant representations of $X$.
Let $\F_{\N X}^\de$ be the class of representations of $\N\T(X)$ that are injective on $X$ and admit an equivariant coaction of $G$.
In \cite{DKKLL20} it is shown that the representation $\N\T(X) \to \cenv(\T_\la(X)^+, G, \de)$ that fixes $X$ is an $\F_{\N X}^\de$-boundary representation.
Furthermore, it is shown that $\cenv(\T_\la(X)^+, G, \de)$ coincides with $A \times_{X, \la} P$. 
These representations differ in general from the augmented Fock representation of Sims--Yeend \cite{SY10} and $\N\O(X)^r$ of \cite{CLSV11}; see for example \cite[Example 3.9]{CLSV11}.
\end{example}

The following example shows that it is not straightforward to develop weak*-analogues of boundary representations.

\begin{remark}
In analogy to the C*-case, Wiersma \cite{Wie16} introduced a weak* tensor product theory for von Neumann algebras.
The corresponding tensor products for von Neumann algebras $M$ and $N$ are taken with respect to injective $*$-representations of $M \odot N$ that are normal on $M$ and $N$.
In \cite{Wie16} it is shown that $(M \otimes_{\max} N)^{**}$ defines a maximal weak* tensor product $\ol{\otimes}_{\textup{w*-}\max}$.
However, the collection of the normal representations of $M \ol{\otimes}_{\textup{w*-}\max} N$ that are injective on $M \odot N$ may not have a boundary element.
An example is given in \cite[Remark 2.6]{Wie16} by considering $N = M'$ for a factor $M$.
Then the multiplication map $m \colon M \odot M' \to \B(H)$ defines a weak* tensor product $\ol{\otimes}_m$, but $M \ol{\otimes}_m M'$ cannot be a normal quotient of $M \ol{\otimes} M'$ when $M$ is not of Type I.
Nevertheless, both $M \ol{\otimes}_m M'$ and $M \ol{\otimes} M'$ give minimal weak*-tensor products being factors.
\end{remark}

\subsection{Co-universality for product systems through the C*-envelope}

We will now show how the C*-envelope of the cosystem on the Fock tensor algebra can be used to answer the co-universality problem for covariant representations.
This first proof we give follows the strategy of \cite{DK18, DKKLL20} through the completely contractive maps of the Fock tensor algebra and an induced coaction by $\ca_\la(P)$.

Recall that Fell's absorption principle for groups requires an intertwining unitary.
We are going to establish a similar result for product systems (actually half of it).
In the case of $X_p = \bC$ for all $p \in P$ we recover \cite[Lemma 5.3]{CD21}.

\begin{proposition}\label{P:inter W}
Let $P$ be a unital subsemigroup of a discrete group $G$ and let $X$ be a product system over $P$.
If $t$ is a representation of $X$, then the isometry
\[
W \colon \F X \otimes_A H \longrightarrow H \otimes \ell^2(P); \xi_q \otimes h \mapsto t(\xi_q)h \otimes \de_q,
\] 
satisfies
\[
W \la(\xi_p) \otimes I_H = (t(\xi_p) \otimes V_p) W \foral p \in P.
\]
\end{proposition}

\begin{proof}
A straightforward computation for $\xi_p, \eta_q \in \F X$ and injectivity of $t$ gives that
\[
\sca{ W \xi_p \otimes h_p, W \xi_q \otimes h_q } 
= 
\de_{p,q} \sca{h_p, t(\xi_p)^* t(\xi_q) h_q}
=
\sca{\xi_p \otimes h_p, \xi_q \otimes h_q}_{\F X \otimes_A H},
\]
Thus $W$ is a well-defined isometry.
We then compute
\[
W (\la(\xi_p) \otimes I_H) (\xi_q \otimes h) = t(\xi_p \xi_q)h \otimes \de_{pq} = (t(\xi_p) \otimes V_p) W (\xi_q \otimes h)
\]
as required, and the proof is complete.
\end{proof}

\begin{proposition}\label{P:cc hom nsa}
Let $P$ be a unital subsemigroup of a discrete group $G$ and let $X$ be a product system over $P$.
If $t$ is an equivariant representation of $X$ that is injective on $A$, then there exists a completely contractive homomorphism
\[
\ol{\alg}\{t(X_p) \mid p \in P\} \longrightarrow \L( \F X ); t(\xi_p) \mapsto \la(\xi_p).
\] 
\end{proposition}

\begin{proof}
First we construct a completely contractive map
\[
\Phi \colon \ol{\alg}\{t(X_p) \otimes V_p \mid p \in P\} \longrightarrow \L( \F X); t(\xi_p) \otimes V_p \mapsto \la(\xi_p).
\]
Let $W$ be the intertwining isometry from Proposition \ref{P:inter W} with final space densely spanned by $t(\xi_q)h \otimes \de_q$.
Hence $WW^*$ is an invariant subspace for $\alg\{t(X_p) \otimes V_p \mid p \in P\}$.
Moreover, Proposition \ref{P:inter W} yields
\[
(t(\xi_p) \otimes V_p) |_{W W^*} = W \left( \la(\xi_p) \otimes I_H \right) W^*,
\]
and injectivity of $t$ on $A$ implies that the map
\[
\otimes \id_H \colon \L( \F X ) \longrightarrow \L( \F X \otimes_A H); S \mapsto S \otimes \id_H
\]
is a faithful $*$-homomorphism \cite[p. 42]{Lan95}.
Thus we can define the following completely contractive homomorphism
\[
\Phi \colon \alg\{t(X_p) \otimes V_p \mid p \in P\} \stackrel{|_{WW^*}}{\longrightarrow}
\alg\{W (\la(X_p) \otimes I_H) W^* \mid p \in P\} \stackrel{\simeq}{\longrightarrow}
\T_\la(X)^+ \otimes I_H \simeq \T_\la(X)^+,
\]
so that $\Phi(t(\xi_p) \otimes V_p) = \la(\xi_p)$.

Next consider the unital completely positive map
\[
\phi \colon \ca_{\max}(G) \longrightarrow \ca_\la(G) \longrightarrow \B(\ell^2(P)) ; u_g \mapsto \la_g \mapsto P_{\ell^2(P)} \la_g |_{\ell^2(P)},
\]
which is multiplicative on the subalgebra of $\ca_{\max}(G)$ generated by the $u_p$ for $p \in P$, with $\phi(u_p) = V_p$.
Then the map $\Phi (\id \otimes \phi) \de_t$ for the coaction $\de_t$ of $t$ satisfies
\[
\Phi (\id \otimes \phi) \de_t t(\xi_p) = \Phi (\id \otimes \phi) (t(\xi_p) \otimes u_p) = \Phi(t(\xi_p) \otimes V_p) = \la(\xi_p),
\]
and gives the required completely contractive homomorphism.
\end{proof}

As a consequence, we obtain a co-action by $\ca_\la(P)$ on the Fock tensor algebra where the target space can be over any equivariant injective representation of $\T_\la(X)$.
The following is the analogue of the key result \cite[Proposition 4.5]{DKKLL20}.

\begin{proposition}\label{P:cocover nsa}
Let $P$ be a unital subsemigroup of a discrete group $G$ and let $X$ be a product system over $P$.
If $\Psi$ is an equivariant $*$-representation of $\T_\la(X)$ that is injective on $A$, then there exists a completely isometric homomorphism
\[
\T_\la(X)^+ \longrightarrow \ca(\Psi) \otimes \ca_\la(P); \la(\xi_p) \mapsto \Psi \la(\xi_p) \otimes V_p.
\] 
In particular, $\ca(\Psi)$ is a C*-cover of the cosystem $(\T_\la(X)^+, G, \de)$.
\end{proposition}

\begin{proof}
First we note that $\Psi \la \otimes V \colon \xi_p \mapsto \Psi\la(\xi_p) \otimes V_p$ defines a representation of $X$. 
Moreover it is equivariant as $\Psi$ is so.
Let $\Phi$ be the completely contractive homomorphism constructed in the proof of Proposition \ref{P:cc hom nsa}, so that $\Phi(\Psi \la(\xi_p) \otimes V_p) = \la(\xi_p)$.
Next consider the unital completely positive map
\[
\phi \colon \ca_{\max}(G) \longrightarrow \ca_\la(G) \longrightarrow \B(\ell^2(P)) ; u_g \mapsto \la_g \mapsto P_{\ell^2(P)} \la_g |_{\ell^2(P)}.
\]
For the coaction $\de_\Psi$ on $\ca(\Psi)$ we directly compute
\[
(\id \otimes \phi) \de_\Psi \Psi(\la(\xi_p)) = (\id \otimes \phi)(\Psi \la(\xi_p) \otimes u_p) = \Psi \la(\xi_p) \otimes V_p.
\]
Therefore, the map $(\id \otimes \phi) \de_\Psi \Psi$ is a completely contractive left inverse for $\Phi$, and thus $\Phi$ is completely isometric.
On the other hand, $\Phi (\id \otimes \phi) \de_\Psi$ is a completely contractive left inverse of $\Psi$, and thus $\Psi$ is completely isometric.
\end{proof}

We next show that the C*-envelope of the cosystem on $\T_\la(X)^+$ is the required terminal object.

\begin{theorem}\label{T:co-un cenv}
Let $P$ be a unital subsemigroup of a discrete group $G$ and let $X$ be a product system over $P$.
Then $\cenv(\T_\la(X)^+, G, \de)$ is a boundary representation for the equivariant injective covariant representations of $X$.
Moreover, we have that
\[
\cenv(\T_\la(X)^+, G, \de) \simeq A \times_{X, \la} P.
\]
\end{theorem}

\begin{proof}
Let $\F_{\cov}$ be the class of equivariant injective covariant representations of $X$.
By definition the relations in $\T_\la(X)$ pass to $\cenv(\T_\la(X)^+, G, \de)$, and so the latter is Fock-covariant and thus it is indeed in $\F_{\cov}$.
Now let $t$ be an element of $\F_{\cov}$.
By Proposition \ref{P:separating} we can assume without loss of generality that $t$ is actually Fock-covariant.
By Remark \ref{R:Exel} we get that $t_{\ast}$ implements a $*$-epimorphism
\[
\Psi \colon \T_\la(X) \simeq \ca_\la(\F\C_G X) \longrightarrow \ca_\la(t_{\ast}(\F\C_G X)),
\]
where $t_{\ast}(\F\C_G X)$ is the induced Fell bundle inside $\ca(t)$.
Let
\[
\la_t \colon \ca(t) \longrightarrow \ca_\la(t_{\ast}(\F\C_G X)).
\]
be the canonical $*$-epimorphism for the Fell bundle $t_{\ast}(\F\C_G X)$, as $t$ admits a coaction of $G$.
By Proposition \ref{P:cocover nsa} we get that the map $\Psi$ defines a C*-cover of the cosystem $(\T_\la(X)^+, G, \de)$, and thus there exists a canonical $*$-epimorphism
\[
\Psi' \colon \ca(t) \longrightarrow (\T_\la(X)^+, G, \de).
\]
Hence we get the following commutative diagram
\[
\xymatrix{
\T_{\cov}^{\fock}(X) \ar[r]^{t_{\ast}} \ar[dr]^{\la_{\ast}} & \ca(t) \ar[r]^{\la_t \phantom{ooooo}} & \ca_{\la}(t_{\ast}(\F\C_G X)) \ar[dr]^{\Psi'} & \\
 & \T_\la(X) \ar[rr]^{\Phi} \ar[ur]^{\Psi} && \cenv(\T_\la(X)^+, G, \de)
}
\]
of canonical $*$-epimorphisms, where $\Phi$ is defined by the definition of the C*-envelope of the cosystem.
This shows that the canonical $*$-representation
\[
\T_{\cov}^{\fock}(X) \stackrel{\Phi \la_{\ast}}{\longrightarrow} \cenv(\T_\la(X)^+, G, \de)
\]
factors through $t_{\ast}$.

For the second part, by Proposition \ref{P:A emb} we have that $A \times_{X, \la} P$ defines an equivariant representation of $\F\C_G X$ that is injective on $A$.
Hence there exists an equivariant $*$-epimorphism
\[
A \times_{X, \la} P \longrightarrow \cenv(\T_\la(X)^+, G, \de)
\]
that fixes $X$, and thus it is injective on $A$.
By Proposition \ref{P:lift A to fpa} this map is injective on $[\S\C_G X]_e$ by being injective on $A$.
Recall that the coaction on $A \times_{X, \la} P$ is normal.
Therefore, by Remark \ref{R:Exel} we have that the $*$-epimorphism is faithful.
\end{proof}

\subsection{Co-universality for product systems through strong covariance}

We can have an independent proof of the co-universality result that goes through the Fell bundle theory.
We have all tools to implement the strategy of Carlsen--Larsen--Sims--Vittadello \cite{CLSV11} and get the existence of the co-universal representation through $A \times_{X, \la} P$.
We also provide a Fell bundle route for Proposition \ref{P:cocover nsa} and complete the connection with the C*-envelope.

In Proposition \ref{P:cocover nsa} we saw how a Fell's absorption principle induces that representations of $\T_\la(X)^+$ that are injective on $A$ are automatically completely isometric.
This works also in reverse by using the Fell bundle technique of \cite[Theorem 3.3]{Exe97}.

\begin{proposition}\label{P:cc hom cstar}
Let $P$ be a unital subsemigroup of a discrete group $G$ and let $X$ be a product system over $P$.
If $t$ is an equivariant representation of $X$ that is injective on $A$, then there exists a completely contractive homomorphism
\[
\ol{\alg}\{t(X_p) \mid p \in P\} \longrightarrow \L( \F X ); t(\xi_p) \mapsto \la(\xi_p).
\] 
\end{proposition}

\begin{proof}
Let $Y_t$ be the Hilbert module associated to $t$ from Subsection \ref{Ss:Bessel}, so that
\[
\sca{\eta_1, \eta_2}_Y := E_{t_\ast}(\eta_1^* \eta_2) \foral \eta_1, \eta_2 \in \sum_p t(X_p),
\]
where $E_{t_\ast}$ is the conditional expectation on $\ca(t)$.
Let $i, j, k \in \{1, \dots, n\}$, and let
\[
t({\xi}^{ij}) := \sum_{p \in P} t(\xi^{ij}_p) \qand {\eta}^k := \sum_{q \in P} t(\eta^k_q)
\]
be finitely supported vectors in $\sum_p t(X_p)$.
Since $t( \xi^{ik} ) \eta_k \in \sum_p t(X_p)$, by Bessel's inequality of Corollary \ref{C:Bessel} we have
\begin{align*}
\| [t({\xi}^{ij}) ] [\eta^k] \| \leq \nor{ [ t({\xi}^{ij}) ] } \cdot \| [\eta^k] \|.
\end{align*}
This yields a well-defined completely contractive map
\[
L \colon \alg\{ t(X_p) \mid p \in P \} \longrightarrow \B Y
\]
by \emph{bounded} operators on $Y$, such that $L(t(\xi_p)) t(\eta_q) = t(\xi_p \eta_q)$.
In particular we see that this map takes values in $\L Y$ with
\[
L(t(\xi_p))^* t( \eta_q)
= 
\begin{cases} 
t(\xi_p^* \eta_q) & \text{if } q \in pP, \\
0 & \text{otherwise}.
\end{cases}
\]

Next we note that $Y_{t}$ is unitarily equivalent to the Fock space $\F X$ by the map
\[
U( \sum_{p} t(\xi_p) ) = \sum_{p} \xi_p.
\]
Indeed, first define $U$ on finite sums; equivariance of $t$ implies that $\{t(X_p)\}_{p \in P}$ is a family of linearly independent spaces in $\ca(t)$, and injectivity of $t$ yields that $U$ is well defined.
Moreover injectivity of $t$ gives that $U$ is an isometry onto $\sum_p t(X_p)$.
Continuity of $E_{t_{\ast}}$ allows to extend the map $U$ to a unitary from $Y_{t}$ onto $\F X$.

Hence we deduce a completely contractive map
\[
\ad_U L \colon \alg\{ t(X_p) \mid p \in P \} \longrightarrow \L(\F X),
\]
such that
\[
\ad_U L( t(\xi_p)) \xi_r = U t(\xi_p) t(\xi_r) = \xi_p \xi_r = \la(\xi_p) \xi_r,
\]
and the proof is complete.
\end{proof}

\begin{proposition}\label{P:cocover cstar}
Let $P$ be a unital subsemigroup of a discrete group $G$ and let $X$ be a product system over $P$.
If $\Psi$ is an equivariant $*$-representation of $\T_\la(X)$ that is injective on $A$, then $\ca(\Psi)$ is a C*-cover of the cosystem $(\T_\la(X)^+, G, \de)$.
Moreover, there exists a completely isometric homomorphism
\[
\T_\la(X)^+ \longrightarrow \ca(\Psi) \otimes \ca_\la(P); \la(\xi_p) \mapsto \Psi \la(\xi_p) \otimes V_p.
\] 
\end{proposition}

\begin{proof}
By applying Proposition \ref{P:cc hom cstar} for $t = \Psi \la|_X$ we have a completely contractive map
\[
\ad_U L \colon \alg\{ \Psi\la(X_p) \mid p \in P \} \longrightarrow \T_\la(X)^+,
\]
such that $\ad_U L(\Psi \la(\xi_p)) = \la(\xi_p)$.
Consequently $\ad_U L$ is a completely contractive left inverse to the completely contractive map $\Psi|_{\T_\la(X)^+}$, making the latter a complete isometry.

For the second part, it suffices to show this for $\Psi = \la$.
As in Proposition \ref{P:cocover nsa}, let the isometry 
\[
W \colon \F X \otimes \ell^2(P) \longrightarrow \F X \otimes \ell^2(P); \xi_r \otimes \de_s \mapsto \xi_r \otimes \de_{r s}.
\]
It has final space densely spanned by $\{\xi_r \otimes \de_{rs} \mid r, s \in P; \xi_r \in X_r\}$, and thus invariant for $\alg\{\la(X_p) \otimes V_p \mid p \in P\}$.
Moreover we directly verify that
\[
W (\la(\xi_p) \otimes I_{\ell^2(P)}) (\xi_r \otimes \de_s) = \xi_p \xi_r \otimes \de_{prs} = (\la(\xi_p) \otimes V_p) W (\xi_r \otimes \de_s).
\]
Hence we can define a completely contractive map $\Phi$ by
\[
\Phi \colon \alg\{\la(X_p) \otimes V_p \mid p \in P\} \stackrel{|_{WW^*}}{\longrightarrow}
\alg\{W (\la(X_p) \otimes I_{\ell^2(P)}) W^* \mid p \in P\} \stackrel{\simeq}{\longrightarrow}
\T_\la(X)^+,
\]
so that $\Phi(\la(\xi_p) \otimes V_p) = \la(\xi_p)$.
It follows that $\Phi$ has a completely contractive left inverse given by $(\id \otimes \phi) \de$, for the unital completely positive map $\phi \colon \ca_{\max}(G) \to \B(\ell^2(P))$ with $\phi(u_g) = P_{\ell^2(P)} \la_g |_{\ell^2(P)}$, and the coaction $\de$ of $\T_\la(X)$.
Thus $\Phi$ is completely isometric.
\end{proof}

We can apply the Fell bundle arguments that appear in the proof of \cite[Theorem 4.1]{CLSV11}.
There is a key underlying idea that can be traced back to Katayama \cite[Proposition 11]{Kat85} and Raeburn \cite[Lemma 3.1]{Rae92}.
It relies on the fact that any coaction induces a reduced coaction to a quotient by an appropriate ideal; and thus that quotient is minimal for the quotient Fell bundle, hence receiving also from the original Fell bundle.

\begin{theorem}\label{T:co-un sc}
Let $P$ be a unital subsemigroup of a discrete group $G$ and let $X$ be a product system over $P$.
Then $A \times_{X, \la} P$ is a boundary representation for the equivariant injective covariant representations of $X$.
Moreover, we have that
\[
A \times_{X, \la} P \simeq \cenv(\T_\la(X)^+, G, \de).
\]
\end{theorem}

\begin{proof}
As in the proof of Theorem \ref{T:co-un cenv} we can reduce to Fock-covariant representations.
Hence, without loss of generality, let $t$ be an equivariant injective Fock-covariant representation of $X$.
Recall from Proposition \ref{P:fpa t} that
\[
\ker t_\ast \bigcap [\T_{\cov}^{\fock}(X)]_e \subseteq \ker q_{\scv} \bigcap [\T_{\cov}^{\fock}(X)]_e,
\]
for the canonical $*$-epimorphism $q_{\scv} \colon \T_{\cov}^{\fock}(X) \to A \times_X P$ that fixes every $X_p$.
Remark \ref{R:induced} yields
\[
t_\ast(\ker q_{\scv}) \bigcap [\ca(t)]_e = t_\ast( \ker q_{\scv} \bigcap [\T_{\cov}^{\fock}(X)]_e ),
\]
since $q_{\scv}$ is equivariant.
Let $q_J$ be the quotient map on $\ca(t)$ by the ideal
\[
J := \sca{t_\ast(\ker q_{\scv}) \bigcap [\ca(t)]_e}.
\]
The goal is to show that
\[
\ker q_J t_\ast \bigcap [\T_{\cov}^{\fock}(X)]_g = \ker q_{\scv} \bigcap [\T_{\cov}^{\fock}(X)]_g \foral g \in G.
\]
It suffices to show this for $g = e_G$, and then the equality for all fibres follows by a standard argument.
With that in hand we deduce that the Fell bundle defined in $\ca(t)/J$ is isomorphic to the strongly covariant bundle $\S\C_G X$.
Consequently, we obtain an equivariant $*$-epimorphism
\[
\ca(t) \longrightarrow \ca(t)/J \longrightarrow A \times_{X, \la} P.
\]

Towards this end first assume that $x \in \ker q_{\scv} \bigcap [\T_{\cov}^{\fock}(X)]_e$.
Then
\[
t_\ast(x) \in t_\ast( \ker q_{\scv} \bigcap [\T_{\cov}^{\fock}(X)]_e ) = t_\ast(\ker q_{\scv}) \bigcap [\ca(t)]_e \subseteq J,
\]
and so $x \in \ker q_J t_\ast \bigcap [\T_{\cov}^{\fock}(X)]_e$.
The reverse inclusion will follow by applying Proposition \ref{P:fpa t} on $q_J t_\ast$.
Towards this end we need to show that $q_J t_\ast$ is injective on $\iota(A)$.
Let $a \in A$ so that $q_J t_\ast(\iota(a)) = 0$, and so $t_\ast(\iota(a)) \in J \bigcap [\ca(t)]_e$.
Since
\begin{align*}
J \bigcap [\ca(t)]_e 
& \subseteq
t_\ast(\ker q_{\scv}) \bigcap [\ca(t)]_e
=
t_\ast( \ker q_{\scv} \bigcap [\T_{\cov}^{\fock}(X)]_e ),
\end{align*}
there exists a $y \in \ker q_{\scv} \bigcap [\T_{\cov}^{\fock}(X)]_e$ such that $t_\ast(\iota(a)) = t_\ast(y)$.
We deduce that
\[
\iota(a) - y \in \ker t_\ast \bigcap [\T_{\cov}^{\fock}(X)]_e \subseteq \ker q_{\scv} \bigcap [\T_{\cov}^{\fock}(X)]_e,
\]
where we applied Proposition \ref{P:fpa t} for the injective Fock-covariant representation $t$.
It follows that $\iota(a) = (\iota(a) - y) + y \in \ker q_{\scv}$, and thus Proposition \ref{P:A emb} yields $a = 0$ as required.

For the second part, Proposition \ref{P:A emb} and Proposition \ref{P:cocover nsa} give a canonical $*$-epimorphism
\[
A \times_{X, \la} P \longrightarrow \cenv(\T_\la(X), G, \de)
\]
that fixes $X$.
On the other hand, the co-universal property of $A \times_{X, \la} P$ provides a canonical $*$-epimorphism in the other direction giving the required $*$-isomorphism.
\end{proof}

\subsection{Applications and remarks}

We close with a set of immediate applications and questions that are related to the theory.
First we observe that co-universality of $A \times_{X, \la} P$ implies a gauge-invariant uniqueness theorem.

\begin{corollary}[GIUT]
Let $P$ be a unital subsemigroup of a discrete group $G$ and let $X$ be a product system over $P$.
Let $t_\ast$ be a representation of $A \times_{X, \la} P$.
Then $t_\ast$ is faithful if and only if $t$ is an equivariant injective representation of $X$.
\end{corollary}

\begin{proof}
The forward direction is immediate as $A \times_{X, \la} P$ shares all these properties.
For the converse, if $t$ satisfies these properties then co-universality of $A \times_{X, \la} P$ gives an inverse to $t_\ast$.
\end{proof}

One may compare with the following corollary, which is a consequence of the properties of $A \times_X P$.

\begin{corollary}[GIUT II]
Let $P$ be a unital subsemigroup of a discrete group $G$ and let $X$ be a product system over $P$.
Let $t$ be a strongly covariant representation of $X$.
Then there is a canonical $*$-isomorphism $\ca(t) \to A \times_{X, \la} P$ if and only if $t$ is injective and equivariant by a normal coaction.
\end{corollary}

\begin{proof}
The forward direction is immediate as $A \times_{X, \la} P$ shares all these properties.
For the converse by universality $t_\ast$ defines a $*$-representation $\S\C_G X$.
Since $t$ is injective we have that $t_\ast$ injective on $[\S\C_G X]_e$ and thus by Remark \ref{R:Exel} there is a canonical $*$-isomorphism $\ca(t) \to A \times_{X, \la} P$.
\end{proof}

Recall that by construction we have a canonical $*$-epimorphism $q_{\scv} \colon \T_{\cov}^{\fock}(X) \longrightarrow A \times_X P$.
By Remark \ref{R:Exel} it descends to a canonical $*$-epimorphism
\[
q_{\scv, \la} \colon \T_\la(X) \longrightarrow A \times_{X, \la} P.
\]
By construction $\ker q_{\scv, \la}$ contains $\la_\ast(\I_{\scv})$.
In particular, due to the (GIUT II), the canonical $*$-epimorphism
\[
\quo{\T_\la(X)}{\la_\ast(\I_{\scv})} \longrightarrow A \times_{X, \la} P
\]
is faithful if and only if it inherits a normal coaction from $\T_\la(X)$.
The reason we are interested in this quotient in the Fock representation is because the ideal $\la_\ast(I_{\scv})$ has been oftenly recognized to give a boundary space.
This still happens in many cases.
The following has been observed for right LCM-semigroups in \cite{DKKLL20}.

\begin{corollary}[Fock boundary quotient]
Let $P$ be a unital subsemigroup of a discrete group $G$ and let $X$ be a product system over $P$.
If $G$ is exact then $\ker q_{\scv, \la} = \la_\ast(\I_{\scv})$, i.e., there is a canonical $*$-isomorphism
\[
\quo{\T_\la(X)}{\la_\ast(\I_{\scv})} \simeq A \times_{X, \la} P.
\]
\end{corollary}

\begin{proof}
Recall that $\T_\la(X)$ admits a normal coaction and that $\la_\ast(\I_{\scv})$ is an induced ideal.
Exactness of $G$ implies that the induced coaction on the quotient is normal as well.
\end{proof}

Hence under exactness we have a hands-on understanding of the algebraic structure of the boundary quotient.

\begin{remark}[Applications I]
The GIUT has been a central tool in the $P = \bZ_+$ case.
Its ample applicability is primarily manifested in graph C*-algebras.
Yet, there are many applications to Cuntz--Pimsner algebras that include a 6-term short exact sequence for computing its K-theory \cite{Kat04}, classification \cite{BTZ18}, and the impact of shift-type equivalences up to strong Morita equivalence.
The reader may refer to \cite{DK18} for a list of results which we will not duplicate here.
The question here is to see to what extent the known theorems in the $\bZ_+$-case find their analogues for general semigroups.
\end{remark}

\begin{remark}[Applications II]
At the same time, recent work on right LCM-semigroups has shown that co-universality results have an immediate impact on Takai duality \cite{DK20, Kat20} and the reduced Hao-Ng problem \cite{DKKLL20}, with interesting interactions with K-theory problems \cite{Sch15}, but also can feed back to the identification of universal C*-algebras \cite{KKLL21}, or in cases can reduce the form of Cuntz-type C*-algebras to generalized crossed products \cite{KKLL21, Kat20, Seh21} as in the $\bZ_+$-case \cite{AEE98}.
The use of $\cenv(\T_\la(X)^+, G, \de)$ is pivotal for obtaining such results, making the step forwards from right LCM-semigroups to general semigroups highly plausible.
\end{remark}

\begin{remark}[Fock tensor algebra]
We have seen that $\T_\la(X)^+$ plays a central role in the co-universality results.
It is a direct generalization of Popescu's noncommutative disc algebra \cite{Pop91} which has been under significant study.
There are plenty of directions one can go from here starting with identifying its completely contractive representations as in \cite{MS98}.
The form of the strong covariance relations may also shed light into hyperrigidity questions, which is resolved by Katsoulis and Ramsey for $P = \bZ_+$ \cite{KR21}.
As they have shown in \cite{KR20}, such a result will then direct feed back to Takai duality and the Hao-Ng problem.
Finally, there is a lot of work for $P = \bZ_+$ in terms of the w*-analogue of the Fock tensor algebra.
Questions include reflexivity, hyperreflixivity as well as a weak*-calculus, with the list being too long to include it here.
\end{remark}


\end{document}